\documentclass[11pt,reqno]{amsart}

\usepackage[utf8]{inputenc}   
\usepackage[T1]{fontenc}

\usepackage{mathtools}  
\usepackage{amssymb}
\usepackage{amsfonts}
\usepackage{bm}         
\usepackage{mathrsfs}   

\usepackage{amsthm}

\usepackage{booktabs}
\usepackage{caption}
\usepackage{tcolorbox}
\usepackage{relsize}
\usepackage[shortlabels]{enumitem}
\usepackage{mdframed}   
\usepackage{circuitikz} 
\usepackage{blindtext}  

\usepackage[top=27mm,bottom=27mm,left=27mm,right=27mm]{geometry}

\usepackage[colorlinks=true,linkcolor=black,citecolor=black,urlcolor=blue]{hyperref}
\usepackage[capitalize,nameinlink,noabbrev]{cleveref}

\newcommand{\fa}{\mathfrak{a}}

\newcommand{\fp}{\mathfrak{p}}

\newcommand{\fq}{\mathfrak{q}}

\newcommand{\Q}{\mathbb{Q}}

\newcommand{\Z}{\mathbb{Z}}
\newcommand{\F}{\mathbb{F}}

\newcommand{\cA}{\mathcal{A}}
\newcommand{\cB}{\mathcal{B}}

\newcommand{\cD}{\mathcal{D}}

\newcommand{\cF}{\mathcal{F}}
\newcommand{\cG}{\mathcal{G}}
\newcommand{\cH}{\mathcal{H}}

\newcommand{\cM}{\mathcal{M}}

\newcommand{\cO}{\mathcal{O}}
\newcommand{\cP}{\mathcal{P}}
\newcommand{\cQ}{\mathcal{Q}}

\newcommand{\cS}{\mathcal{S}}

\DeclareMathOperator{\Norm}{N}

\DeclareMathOperator{\Gal}{Gal}

\DeclareMathOperator{\Cl}{Cl}

\newcommand{\house}[1]{%
  \setbox0=\hbox{$#1$}%
  \vrule height \dimexpr\ht0+1.4pt width .4pt depth \dp0\relax
  \vrule height \dimexpr\ht0+1.4pt width \dimexpr\wd0+2pt depth \dimexpr-\ht0-1pt\relax
  \llap{$#1$\kern1pt}%
  \vrule height \dimexpr\ht0+1.4pt width .4pt depth \dp0\relax
}

\newtheorem{theorem}{Theorem}[section]

\newtheorem{lemma}[theorem]{Lemma}
\newtheorem{proposition}[theorem]{Proposition}

\newtheorem{conjecture}[theorem]{Conjecture}

\theoremstyle{definition}

\theoremstyle{remark}

\newtheorem*{theorem*}{Theorem}
\newtheorem*{lemma*}{Lemma}
\newtheorem*{proposition*}{Proposition}
\newtheorem*{corollary*}{Corollary}
\newtheorem*{conjecture*}{Conjecture}
\newtheorem*{definition*}{Definition}
\newtheorem*{example*}{Example}
\newtheorem*{remark*}{Remark}
\newtheorem*{question*}{Question}
\newtheorem*{problem*}{Problem}

\crefname{theorem}{theorem}{theorems}
\Crefname{theorem}{Theorem}{Theorems}
\crefname{equation}{equation}{equations}
\Crefname{equation}{Equation}{Equations}

\title[Statistics of the Genus Number of  $S_3 \times C_q$ and $D_4$-fields]
{Statistics of the Genus Number of  $S_3 \times C_q$ and $D_4$-fields}

\author{ANUP B. DIXIT}
\address{The Institute of Mathematical Sciences, Chennai, 
  IV Cross Road, CIT Campus, Taramani, Chennai - 600 113,
  Tamil Nadu, India.}
\email{anupdixit@imsc.res.in}

\author{Sunil Kumar Pasupulati}
\address{UM-DAE Centre for Excellence in Basic Sciences, University of Mumbai, Vidyanagari, Mumbai, India-400098 }
\email{sunilkumarpasupulati@gmail.com, sunil.pasupulati@cbs.ac.in}


\begin{document}

\begin{abstract}
    The genus number of a number field is a fundamental invariant which measures the contribution of ramification to its ideal class group. In this paper, we establish the statistics for the genus number for $S_3\times C_q$-fields for $q\neq 3$ a prime number, $D_4$-fields and pure quartic fields. We also obtain precise results on the average and higher moments of the genus distribution within the family of $S_3\times C_q$-fields. Finally, based on heuristics, we formulate a conjecture identifying families for which one should expect the genus density to be zero, i.e., only a density zero subset of fields in the family attains any fixed genus number.
\end{abstract}
    	\maketitle

\section{\bf Introduction}
\medskip

Let $K$ be a number field. The genus field $K^*$ of $K$ is defined as the maximal abelian extension of $K$ unramified at all finite places of $K$ and is obtained by composing an absolute abelian field $k^*/\Q$, i.e., $K^* = k^* K$. Essentially, $K^*$ can be thought of as part of the Hilbert class field arising from the abelian extension of $\Q$. The genus number $g_K$ is defined as the degree $[K^*:K]$. Genus theory traces back to Gauss’s foundational work on binary quadratic forms, which was further developed by Hasse, Leopoldt, Fr\"{o}hlich, Ishida and others. By construction, $g_K$ divides $h_K^+$, the narrow class number of $K$. Unlike the case of the class number $h_K$, the genus number $g_K$ only depends on the ramified primes of $K$ and their local behaviour. This makes it a much more tractable invariant, which still captures rich arithmetic data. The reader is referred to \cite{Ishidabook} for a thorough exposition of genus theory.\\

In recent years, the distribution of genus numbers has emerged as an important theme in arithmetic statistics. One typically studies genus numbers across families of number fields of fixed degree, ordered by discriminant. For quadratic fields $K$, with absolute discriminant $|\cD_K|$, a classical formula of Gauss gives
\begin{equation*}
    g_K = 2^{\omega(|\cD_K|)-1},
\end{equation*}
where $\omega(n)$ denotes the number of distinct prime divisors of $n$. This immediately implies that quadratic fields with any fixed genus number have density zero. Beyond the quadratic case, Frei, Loughran, and Newton \cite{FDR23} established asymptotic formulas for the average genus number in families of abelian extensions, showing again that every fixed genus value occurs with density zero. Here, they count number fields by their conductor, as opposed to the discriminant. In striking contrast, McGown and Tucker \cite{cubicwithgenusnumber1} proved that $96.23\%$ of cubic fields have genus number $g_K=1$. In fact, they also showed that each value of the form $g_K=3^l$ for $l\geq 0$ occurs with a positive density among cubic fields. An analogous phenomenon was recently demonstrated for quintic fields by McGown, Thorne, and Tucker \cite{MTT23}, who showed that a large proportion of quintic fields satisfy $g_K=1$ and that every value $g_K = 5^l$ appears with positive density. \\

This contrast naturally leads to a fundamental question: For which families of number fields should one expect a positive proportion of fields with a prescribed genus number, and for which families does every fixed genus number occur only on a zero-density subset. To address this, we study genus statistics over $S_3\times C_q$ fields (with $q\neq 3$ is prime), quartic $D_4$-fields and pure quartic fields. We then propose a conjecture that delineates which families of number fields are likely to exhibit a zero-density phenomenon for any fixed genus number.\\

Denote by $\cD_K$ the discriminant of $K/\Q$. For $L/K$ an extension of number fields, let $\cD_{L/K}$ denote the relative disriminant ideal. Let $\widetilde{K}$ denote the Galois closure of $K/\Q$. For a finite group $G$, let
\begin{equation*}
    \cM(G, n,X):=\{K/\mathbb{Q} : [K:\Q]=n,\, \operatorname{Gal}(\widetilde{K}/\mathbb{Q})\simeq G,\; |\cD_K|\le X\}
\end{equation*}
and $\cM(G,n):= \cM(G,n, \infty)$. When $G\cong S_3\times C_m$, it is known due to Masri, Thorne, Tsai and Wang \cite{MasriThorneTsaiWang2020} (also see \cite{wang21}) that
\begin{equation*}
    |\cM(S_3\times C_m, 3m, X)| \sim c(m)\, X^{1/m},
\end{equation*}
where $c(m)$ is a constant. This coincides with the prediction of the famous Malle's conjecture \cite{Mal04}. For $m=2$, we obtain the finer estimate (see Proposition \ref{Count-S_3xC_2}), namely,

\begin{equation*}
    |\cM(S_3\times C_2, 6, X)| = \lambda \, X^{1/2} + O\left(X^{5/12+\epsilon}\right),
\end{equation*}
where
 \begin{align*}
   \lambda = \frac{1}{3\zeta(3)}\left[ \frac{1+ 4^{-3/2} + 8^{-3/2}}{1+2^{-3/2}}\,\, \frac{\zeta(3/2)}{\zeta(3)} \, - \left(1 + \frac{3}{7\cdot 4^{3/2}} + \frac{6}{7\cdot 8^{3/2}}\right)\prod_{\substack{p\text{ prime}\\ p \geq 3}} \left(1 + \frac{p+1}{p^{3/2}(p^2+p+1)}\right) \right] .
\end{align*}
For an $S_3\times C_2$-field $K$, the genus number is of the form $g_K=2^k 3^{\ell}$, where $k$ and $\ell$ are non-negative integers. Define
\begin{equation*}
    \cQ_{k,\ell}^2(X):= \{K\in \cM(S_3\times C_2,6,X) : g_K = 2^k 3^{\ell}\}. 
\end{equation*}
 Suppose $ T_{\ell} $ 
denote the set of squarefree integers $ n $ that are coprime to 3 and have exactly $ \ell $ prime factors $ p$ satisfying $ p \equiv 1 \pmod{3} $. We also set $T_{-1}$ as the empty set. Our first theorem is as follows. 
\begin{theorem}\label{S_3-C_2-theorem}
    For integers $k,\ell\geq 0$,
    \begin{align*}
        |\cQ_{k,\ell}^2(X)| = \frac{A(k) B(\ell)}{\zeta(2)}\sqrt{X} + O\left(X^{8/17 +\epsilon}\right),
    \end{align*}
where 
    $$
        A(k) =  \sum_{\substack{ F= \Q(\sqrt{r})  \\\omega(\cD_F) = k+1}}  \frac{1}{|\cD_F|^{3/2}}
        \bigg(1 - \prod_{q \mid \cD_F}   \frac{q+1}{q^2 + q + 1}\bigg)
    $$
and 
    $$
        B(\ell) = 
    \frac{29}{81} \sum_{f \in T_{\ell}} \prod_{p \mid f} \frac{1}{p(p+1)}
   + \frac{1}{324} \sum_{f \in T_{\ell-1}} \prod_{p \mid f} \frac{1}{p(p+1)}.
    $$
\end{theorem}

This implies that a positive proportion of $S_3\times C_2$ fields attain every genus number which is admissible (see ~\Cref{Table-1}).
\begin{table}[t]
\centering
\small
\begin{tabular}{|c|c|c|c|}
\hline
$k$ & $\ell$ &
$A(k)B(\ell)/\zeta(2)$ &
Proportion in $S_3\times C_2$-fields\\[2pt]
\hline

0 & 0 & 0.0856846 & 0.451847 \\
0 & 1 & 0.0147267 & 0.0776774 \\
0 & 2 & 0.000899445 & 0.00474258 \\
0 & 3 & 0.0000272822 & 0.000143885 \\
1 & 0 & 0.0196851 & 0.103812 \\
1 & 1 & 0.00338364 & 0.0178464 \\
1 & 2 & 0.000206575 & 0.00108961 \\
1 & 3 & 0.00000626954 & 0.0000330577 \\

2 & 0 & 0.00208768 & 0.0110099 \\
2 & 1 & 0.000358938 & 0.00189273 \\
2 & 2 & 0.0000219153 & 0.00011556 \\
2 & 3 & 0.000000664862 & 0.00000350598 \\

3 & 0 & 0.000136073 & 0.000717728 \\
3 & 1 & 0.0000233866 & 0.000123385 \\
3 & 2 & 0.00000142840 & 0.00000753326 \\
3 & 3 & 0.0000000433452 & 0.000000228552 \\

\hline
\end{tabular}
\caption{}
\label{Table-1}
\end{table}
Analogously, we consider the family of $S_3\times C_q$ fields for a prime $q\geq 5$. In this case, the genus number is of the form $g_K= 3^{\ell} q^k$, where $\ell$ and $k$ are non-negative integers. Define

\begin{equation*}
    \cQ^q_{k,\ell}(X):= \{K\in \cM(S_3\times C_q,3q,X) : g_K = 3^{\ell} q^{k}\}. 
\end{equation*}

\begin{theorem}\label{S_3-C_q-theorem}
    For any prime $q\geq 5$ and $k,l\geq 0$, there exists a constant $C_q(k,l)>0$, explicitly described in \eqref{implied-constant-1}, such that
    \begin{equation*}
        |\cQ^q_{k,\ell}(X)| = C_q(k,\ell) \, X^{1/q} + O\left(X^{16/17q}\right).
    \end{equation*}
\end{theorem}

Recently, Yamada \cite{Yam26} derived second order terms for the genus statistics of $S_3$-fields. It would be natural to seek analogous second order terms in \Cref{S_3-C_2-theorem,S_3-C_q-theorem} using similar methods. The secondary term in the asymptotics for counting such number fields has already been computed by Wang \cite{wang2017secondaryterm}.\\

We next turn to the family of $D_{4}$-fields. Cohen, Diaz y Diaz, and Olivier \cite{CDO02} proved that the number of $D_{4}$-fields $K$ with $|\cD_K| \leq X$ is given by
\begin{equation*}
    CX + O\left(X^{3/4+\varepsilon}\right),
\end{equation*}
where $C$ is explicitly described. More recently, McGown and Tucker \cite{McTu24} sharpened this result by improving the error term to $O(X^{5/8 +\varepsilon})$. If $K$ is a $D_4$-field, then its genus number is of the form $g_K = 2^a$. We obtain the statistics for genus number over $D_4$-fields as follows.\\ 

\noindent
For a non-zero integer $d$, let $k=\Q(\sqrt{d})$ and define
\begin{equation*}
    C(d):= \prod_{p\text{ splits in }k} \left(1-\frac{1}{p^2}\right)^2 \, \prod_{p \text{ inert in } k}\left(1-\frac{1}{p^2}\right) \prod_{p\text{ ramifies in }k}\left(1-\frac{1}{p}\right).
\end{equation*}
\begin{theorem}\label{thm:D4-genus1-upper}
For every $a\in \Z^+\cup\{0\}$, as $X\to\infty$
\begin{equation*}
    \#\left\{ K\in \cM(D_4,4,X): g_K= 2^a\right\} \sim \delta_a X,
\end{equation*}
where $\delta_a>0$. For $a=0$, we have
\[
\delta_0 = \sum_{\substack{d = \pm p,\ p\ \text{prime}\\ 
d \equiv 1 \pmod{4}}}
\frac{2^{\,u(d)-1}\, C(d)}{d^2}, 
\]
where $u(d)=1$ if $d>0$ and $0$ if $d<0$. 
\end{theorem}

In all the families considered above, we observe that every admissible genus number is realized by a positive proportion of number fields. In contrast, as noted in the introduction, it is known from the work of Frei, Loughran and Newton \cite{FDR23} that for abelian number fields, the density of fields with any prescribed genus number is zero. Motivated by this dichotomy, we propose the following conjecture aimed at identifying families of number fields that exhibit the zero-density phenomenon for genus numbers.\\

\begin{conjecture}\label{conj-1}
Let $G$ be a finite group of order $n$, and let $\cM'(G,n,X)$ denote the subset of $\cM(G,n, X)$ consisting of number fields $K$ that admit a tower
\begin{equation*}
\Q=K_1\subset K_2\subset\cdots\subset K_{k+1}=K,
\end{equation*}
in which each extension $K_{j+1}/K_j$ is abelian for all $1\leq j \leq k$. Then, for any positive integer $t$, as $X \to \infty$, the set of fields $K \in \cM'(G,n,X)$ with $g_K=t$ has density zero.
\end{conjecture}

This conjecture encompasses a broad class of number fields that are, in a sense, close to being abelian. The underlying principle is that the closer a family is to being abelian, the more strongly one expects it to exhibit a zero-density phenomenon for the genus number. In Section \ref{heuristics}, we provide heuristic evidence for this behavior in the cases of octic $D_4$-fields $\cM(D_4, 8)$ and $\cM(C_7\rtimes C_3, 21)$, providing further support for the conjecture.\\ 

Another family that may be regarded as close to abelian is pure fields. A pure quartic field is of the form $K = \mathbb{Q}(\sqrt[4]{a})$, where $a$ is an odd integer with $a \neq \pm 1$. In this setting, we also investigate genus statistics within a family of pure quartic fields, which forms a subfamily of $D_4$-fields. Define the family
\begin{equation*}
    \cB(X):= \left\{
K = \mathbb{Q}(\sqrt[4]{a})  : \; \; \text{ if }\,\,  p^{\nu} \parallel a \text{ for a prime } p, \text{ then } 
\gcd(\nu,4) = 1, |\cD_K|\leq X
\right\}.
\end{equation*}

We show that $|\cB(X)|$ is of order $X^{1/3}$ as $X \to \infty$. In contrast to the case of $D_4$-fields, this subfamily exhibits a zero-density phenomenon for every admissible genus number. More precisely, we prove the following.
\begin{theorem}\label{pure-quartic}
For any $a\in\Z^+\cup\{0\}$, we have
    \begin{equation*}
\lim_{X\to\infty} \frac{\#\{K\in \cB(X) : g_K = 2^a\}}{|\cB(X)|} = 0.
\end{equation*}
\end{theorem}

\medskip

The paper is organized as follows. In Section \ref{prelims}, we recall basic results from genus theory and establish several preliminary lemmata. In Section \ref{S_3xC_2}, we enumerate $S_3\times C_2$-fields and prove Theorem \ref{S_3-C_2-theorem}. The proofs of Theorems \ref{S_3-C_q-theorem} and \ref{thm:D4-genus1-upper} are presented in Sections \ref{S_3xC_q} and \ref{D_4}, respectively. Section \ref{pure-impure} is devoted to genus statistics for pure quartic fields. In Section \ref{higher-moments}, we study the mean and higher moments of the genus numbers for $S_3\times C_q$-fields. Finally, in Section \ref{heuristics}, we discuss heuristic evidence in support of Conjecture \ref{conj-1}.

\medskip

\section{\bf Preliminaries}\label{prelims}
\medskip
Let $K$ be an algebraic number field and $K^*$ denote its genus field. We begin by recalling a result by Ishida \cite{Ish74} (also see \cite[Theorem~3]{Ishidabook}) on the decomposition of the genus field.
\begin{theorem}[Ishida]\label{lemma:k1k2decomp}
Let $K/\mathbb{Q}$ be a number field.  
For each odd prime $q$, write
\[
    q\mathcal{O}_K = \mathfrak{q}_1^{e_1}\cdots \mathfrak{q}_f^{e_f},
    \qquad
    e(q):=\gcd(e_1,\dots,e_f).
\]
Let $K^{*}$ be the genus field of $K$, and write $K^{*}=k^{*}K$, where $k^{*}$ is the maximal
abelian subfield of $K^{*}$ over $\mathbb{Q}$.  Then
\[
    k^{*} = k_1^{*} k_2^{*},
\]
where $k_1^{*}$ and $k_2^{*}$ are linearly disjoint, with $k_1^{*}$ accounting for the contributions of tamely ramified primes and $k_2^{*}$ accounting for those of wildly ramified primes. In particular,
\[
    [k_1^{*}:\mathbb{Q}]
    = \prod_{\substack{p\mid\mathcal{D}_K \\ p ~\text{ is tamely ramified }}}
      \gcd\bigl(e(p),\,p-1\bigr).
\]
\end{theorem}

We also recall the following result, which is due to Gauss for quadratic fields and Fr\"{o}hlich \cite{Fr59} for cubic fields.
\begin{theorem}[Gauss-Fr\"{o}hlich]\label{genus:quad-cubic}
    Let $K$ be a number field with genus number $g_K$.
    \begin{enumerate}
        \item If $K = \mathbb{Q}(\sqrt{m})$, where $m$ is square-free, then
        \begin{align*}
              g_K = 2^{t - 1},
        \end{align*}
         where $t$ is the number of distinct prime divisors of $\cD_K$.\\

         \item Let $K = \mathbb{Q}(\alpha)$ be a cubic number field with a primitive element $\alpha$, and let $D_f$ be the discriminant of the minimal polynomial $f(X)$ of $\alpha$. Then, 
        $$
            g_K =
            \begin{cases} 
                3^{e}, & \text{if } K \text{ is not cyclic over } \mathbb{Q}, \\
                3^{e-1}, & \text{if } K \text{ is cyclic over } \mathbb{Q},
            \end{cases}
        $$
        where $e$ denotes the number of primes $p$ such that $p$ is totally ramified in $K$ and $\left( \frac{D_f}{p} \right) = 1$. Here,  $\left( \frac{D_f}{p} \right)$ is the Legendre symbol.        
    \end{enumerate}
\end{theorem}

Let $ K $ and $ L $ be number fields with degrees $ n = [K:\mathbb{Q}] $ and $ m = [L:\mathbb{Q}] $, and let $ KL $ denote their compositum. Let $ \mathcal{O}_K $, $ \mathcal{O}_L $, and $ \mathcal{O}_{KL} $ be their respective rings of integers.

Then
\[
\cD_{KL} = \cD_K^{m} \cdot \cD_L^{n} \cdot \left[\mathcal{O}_{KL} : \mathcal{O}_K \mathcal{O}_L \right]^2.
\]
In particular, if $ K $ and $ L $ are linearly disjoint over $ \mathbb{Q} $ and $ \mathcal{O}_{KL} = \mathcal{O}_K \mathcal{O}_L $, then
\begin{equation}\label{compositum:disc}
    \cD_{KL} = \cD_K^{m} \cdot \cD_L^{n}.
\end{equation}
We also recall the following lemma on the genus field of a compositum of two number fields (see [Proposition 2 in \cite{Ishidabook}).
\begin{lemma}\label{compositum:genusfields}
Let $K$ be composite of   two number fields  $K_1$ and $K_2$, such that $[K_1 : \mathbb{Q}] = n_1$ and $[K_2 : \mathbb{Q}] = n_2$. Let $K_1^*$ and $K_2^*$ be the genus fields of $K_1$ and $K_2$, respectively.  If $n_1$, $n_2$ are coprime,   then the genus field $K^*$ of $K$ is $
K^* = K_1^*K_2^*$.
\end{lemma}
\noindent
The genus number can now be evaluated for sextic fields, can now be evaluated as follows.
\begin{lemma}\label{genusnumber:sixdegree}
Let $K$ be a cubic number field and $d$ be a square-free integer. For the extension $L = K(\sqrt{d})$, the genus number
\[
g_L =  
\begin{cases}  
3^s 2^{t-1}, & \text{if $K$ is not cyclic over $\mathbb{Q}$}, \\  
3^{s-1} 2^{t-1}, & \text{if $K$ is cyclic over $\mathbb{Q}$},  
\end{cases}  
\]  
where $s$ is the number of primes $p$ totally ramified in $K$ satisfying $\left( \frac{\cD_K}{p} \right) = 1$, and $t$ is the number of distinct prime divisors of the discriminant of $F = \mathbb{Q}\left(\sqrt{d}\right)$.  
\end{lemma}  

\begin{proof}  
Let $F = \mathbb{Q}\left(\sqrt{d}\right)$. Since $[K : \mathbb{Q}] = 3$ 
and $[F : \mathbb{Q}] = 2$, by  Lemma \ref{compositum:genusfields}, 
the genus field of $L$ is $K^* F^*$, and the genus number  
$$
g_L = [K^* F^* : KF] = [K^* : K][F^* : F]=g_Kg_F.  
$$  
Now, using Theorem \ref{genus:quad-cubic}, we have the lemma.
\end{proof}

\begin{lemma}\label{disjointness}
   Let $K$ be a cubic number field and $F$ a quadratic number field. If  $ \cD_F \nmid \cD_K$, then $ K $ and $ F $ are linearly disjoint over $ \mathbb{Q} $, i.e., $ \widetilde{K} \cap F = \mathbb{Q} $.
\end{lemma}

\begin{proof}
Assume the contrary, that $ \widetilde{K} \cap F \neq \mathbb{Q} $. Since $[F:\mathbb{Q}] = 2$, it follows that $F \subset \widetilde{K}$, and hence $ \cD_F \mid \cD_{\widetilde{K}} $. Recall that for a quadratic field $K=\Q(\sqrt{m})$, where $m$ is squarefree, the discriminant is given by
\begin{equation*}
\cD_K =
\begin{cases}
m, & \text{if } m \equiv 1 \pmod{4}, \\
4m, & \text{if } m \equiv 2 \text{ or } 3 \pmod{4}.
\end{cases}
\end{equation*}
Moreover, since $ \widetilde{K} $ is the Galois closure of $K$, it is generated over $\mathbb{Q}$ by the conjugates of a primitive element of $K$. It follows that $ \cD_{\widetilde{K}} \mid \cD_K^{\ell} $ for some integer $\ell \geq 1$. Consequently,
$$ 
    \cD_F \mid \cD_{\widetilde{K}} \mid \cD_K^{\ell}. 
$$
Since the odd part of $\cD_F$ is squarefree and $\cD_K \equiv 0 \text{ or } 1 \pmod{4}$, we deduce that $\cD_F \mid \cD_K$, which yields a contradiction.
\end{proof}
\medskip

We also recall the following theorem of Ishida \cite{Ish75}, which states that for any extension $K/\mathbb{Q}$ of prime degree $q$, the genus number $g_K$ is necessarily a power of $q$.

\begin{lemma}\label{lemma:genus:number:degree:p:extensions}
Let $K/\mathbb{Q}$ be a number field of prime degree $q$.  Set
\begin{equation*}
    t
    = \#\bigl\{\, p\ \text{prime} : 
       p \text{ is totally ramified in } K,\ 
       p \equiv 1 \!\!\pmod q \,\bigr\}.
\end{equation*}
We say $K$ satisfies condition $(\dagger)$ if there exists a primitive element in $\cO_K$ whose minimal polynomial $X^q+a_1 X^{q-1} +\ldots+ a_q$ is Eisenstein $\pmod q$ and satisfies the congruence condition
\begin{equation*}
    a_{2}\equiv a_{3}\equiv\cdots\equiv a_{q-1}
    \equiv a_{1}+a_{q}\equiv 0 \pmod{q^{2}}.
\end{equation*}

\begin{enumerate}

\item[\textup{(1)}]
If $K/\mathbb{Q}$ is non-cyclic, then
\begin{equation*}
    g_K =
    \begin{cases}
        q^{\,t+1} & \text{if }(\dagger)\text{ holds and $q$ is totally ramified in $K$},\\[3pt]
        q^{\,t},  & \text{otherwise}.
    \end{cases}
\end{equation*}
\item[\textup{(2)}]
If $K/\mathbb{Q}$ is cyclic, then
\begin{equation*}
    g_K =
    \begin{cases}
        q^{\,t}   & \text{if $q$ is totally ramified in $K$},\\[3pt]
        q^{\,t-1} & \text{otherwise}.
    \end{cases}
\end{equation*}
\end{enumerate}
\end{lemma}

\noindent

\medskip

\section{\bf Genus statistics on $S_3\times C_2$-fields}\label{S_3xC_2}
\medskip

Consider the set of sextic fields which are a compositum of cubic extensions and quadratic extensions with the discriminant of the quadratic field not dividing that of the cubic field, i.e.,
\[
\mathcal{S} = \left\{ L = KF \ \bigg| \ [K : \mathbb{Q}] = 3, \ [F : \mathbb{Q}] = 2 \ \text{and} \,\, \cD_F \nmid  \cD_K \right\}.
\]

We first show that a generic sextic field $L/\Q$ with $\Gal(\widetilde{L}/\Q)\cong S_3\times C_2$ is in $\cS$. Towards this, we obtain an upper bound for its complement set as follows.
\begin{proposition}\label{most-S_3-C_2}
    Let $\eta(X)$ denote the number of $S_3\times C_2$-extensions $L/\Q$ such that $L=KF$ the compositum of a cubic extension $K$ and a quadratic extension $F$ with $\cD_F\mid \cD_K$. Then $\eta(X)\ll X^{1/3}$. 
\end{proposition}
\begin{proof}
    Since 
    $$
        \cD_L = \cD_{KF} \geq \cD_K^2 \cD_F^3,
    $$
    we have
    \begin{align*}
       \eta(X) = \#\bigg\{L=KF \,\bigg| \Gal(L/\Q)\cong &S_3\times C_2, [K:\Q]=3, [F:\Q]=2, \cD_F\mid \cD_K, |\cD_L|\leq X\bigg\}\\
       &\leq \sum_{Y\leq X^{1/2}} \sum_{\substack{Z\leq \left(\frac{X}{Y^2}\right)^{1/3}\\\\ Z\mid Y}} 1,
    \end{align*}
    where $Y$ runs over $\cD_K$'s and $Z$ over $\cD_F$'s. Hence,
    \begin{align*}
        \eta(X)  \leq \sum_{Y\leq X^{1/2}} \bigg\lfloor \frac{(X/Y^2)^{1/3}}{Y}\bigg\rfloor = \sum_{Y\leq X^{1/2}} \frac{X^{1/3}}{Y^{5/3}}= O(X^{1/3}).
    \end{align*}
    
\end{proof}
\medskip

\subsection{Counting $S_3\times C_2$-fields by discriminant}
We now deduce an asymptotic formula for sextic fields in $\cS$ with discriminant $\leq X$ and show that it is $\gg X^{1/2}$. This renders the fields of the type in the Proposition \ref{most-S_3-C_2} to have density zero.

\begin{proposition}\label{Count-S_3xC_2}
   Let $\cS(X) := \Bigl\{ L \in \mathcal{S} \ \Bigm|\ |\cD_L| \leq X \Bigr\}$. Then
    $$
      |\cS(X)| = \lambda\, X^{1/2}  \;+\; O\!\left(X^{5/12}\right),
    $$
    where
\begin{align*}
   \lambda = \frac{1}{3\zeta(3)}\left[ \frac{1+ 4^{-3/2} + 8^{-3/2}}{1+2^{-3/2}}\,\, \frac{\zeta(3/2)}{\zeta(3)} \, - \left(1 + \frac{3}{7\cdot 4^{3/2}} + \frac{6}{7\cdot 8^{3/2}}\right)\prod_{\substack{p\text{ prime}\\ p \geq 3}} \left(1 + \frac{p+1}{p^{3/2}(p^2+p+1)}\right) \right] .
\end{align*}

\end{proposition}
\begin{proof}
    
Write
\begin{align*}
       \cS(X) &= \#\left\{ KF \ \middle| \ \cD_K^2 \cD_F^3 \leq x ,\,\, \cD_F \mid \cD_K\right\} \\
    &=  \sum_{\substack{ F=\Q(\sqrt{r}) \\ \cD_F  \leq X^{1/3}}} N_{\cD_F}\left( \sqrt{ \frac{X}{\cD_F^3}}\right),
\end{align*}
where  $N_r(X)$ denotes the number of isomorphism classes of cubic fields $K$ such that $|\cD_K| \leq X$ and $r \nmid \cD_K$. If $N(X)$ denotes the number of cubic fields $K$ with $|\cD_K|\leq X$, then from \cite[Theorem 1.1]{countingcubicfields}, we know that 
$$
    N(X) = \frac{1}{3\zeta(3)} \, X + O\left(X^{5/6}\right).
$$
Denote by $M_r(X)$ the number of cubic fields $K$ with $|\cD_K|\leq X$ such that $r\mid \cD_K$.  By \cite[Theorem 1.3]{countingcubicfields}, we have 
\[
M_r(X) =  \frac{1}{3\zeta(3)} \left(\prod_{p\mid r}\frac{p+1}{p^2 + p +1}\right)\, X + O\left(X^{5/6}\right).
\]
Hence,
$$
    N_r(X) = N(X) - M_r(X) = \frac{1}{3\zeta(3)} \left(1-\prod_{p\mid r}\frac{p+1}{p^2 + p +1}\right)\, X + O\left(X^{5/6}\right).
$$
Therefore,
\begin{align*}
\cS(X)
  &= 
     \sum_{\substack{F=\Q(\sqrt{r}) \\ \cD_F \le X^{1/3}}}
        N_{\cD_F}\!\left( \sqrt{\frac{X}{\cD_F^{3}}} \right)\\[2pt]  
  &= \frac{1}{3\zeta(3)}\,X^{1/2}
     \left(
        \sum_{\substack{F=\Q(\sqrt{r}) \\ |\cD_F| \le X^{1/3}}}
          \frac{1}{|\cD_F|^{3/2}}
        \;-\;
        \sum_{\substack{F=\Q(\sqrt{r}) \\ |\cD_F| \le X^{1/3}}}
          \frac{1}{|\cD_F|^{3/2}}
          \prod_{p\mid\cD_F} \frac{p+1}{p^{2}+p+1}
     \right)\\[2pt]
     &\hspace{2cm}+ O\!\left(
         X^{5/12}
         \sum_{\substack{F=\Q(\sqrt{r}) \\ \cD_F \le X^{1/3}}}
           \frac{1}{|\cD_F|^{5/4}}
       \right).
\end{align*}
For $k>1$, note that
\begin{align*}
    \sum_{\substack{F=\Q(\sqrt{r})}} \frac{1}{|\cD_F|^{k}} = \sum_{\substack{r\text{ squarefree}\\r\equiv 1\bmod 4}} \frac{1}{|r|^k} + \sum_{\substack{r\text{ squarefree}\\r\equiv 2,3\bmod 4}} \frac{1}{(4|r|)^k}.
\end{align*}
Since $r$ runs over all integers $\neq 0,1$. Note that if $r\equiv 1\bmod 4$, then $-r\equiv 3\bmod 4$. Thus, the above sum can be written as
\begin{align*}
    \sum_{\substack{F=\Q(\sqrt{r})}} \frac{1}{|\cD_F|^{k}} &= \sum_{\substack{d\text{ squarefree}\\ d> 1, 2\nmid d}} \frac{1}{d^k} +  \sum_{\substack{d\text{ squarefree}\\ d\geq 1, 2\nmid d}} \frac{1}{(4d)^k} +  \sum_{\substack{d\text{ squarefree}\\ d\geq 1, 2\mid d}} \frac{1}{(4d)^k}\\
    & =  \left(\sum_{\substack{d\text{ squarefree}\\ d\geq 1, 2\nmid d}} \frac{1}{d^k}\right) \left(1+ \frac{1}{4^k} + \frac{1}{8^k}\right)-1\\
    &= \left(\prod_{\substack{p \text{ prime}\\p\geq 3}}  \left(1+ \frac{1}{p^k}\right)\right)\left(1+ \frac{1}{4^k} + \frac{1}{8^k}\right)-1\\[2pt]
    & = \frac{\zeta(k)}{\zeta(2k)} \,\frac{1+ 4^{-k} + 8^{-k}}{1+2^{-k}} -1.
\end{align*}
Similarly, by considering the positive and negative discriminants, we derive
\begin{align*}
    \sum_{F=\Q(\sqrt{r})} \frac{1}{\cD_F^k} \prod_{p\mid \cD_F} \frac{p+1}{p^2+p+1}  &=  \sum_{\substack{d\text{ squarefree}\\ d> 1, 2\nmid d}} \frac{1}{d^k} \prod_{p\mid d} \frac{p+1}{p^2+p+1} +  \sum_{\substack{d\text{ squarefree}\\ d\geq 1, 2\nmid d}} \frac{1}{(4d)^k} \prod_{p\mid d} \frac{p+1}{p^2+p+1} \left(\frac{3}{7}\right)\\
    &\hspace{2cm} +  \sum_{\substack{d\text{ squarefree}\\ d\geq 1, 2\mid d}} \frac{1}{(4d)^k} \prod_{p|d} \frac{p+1}{p^2+p+1}\\
    & = \left[\prod_{\substack{p\text{ prime}\\ p \geq 3}} \left(1 + \frac{p+1}{p^k(p^2+p+1)}\right) \right] \left(1 + \frac{3}{7\cdot 4^k} + \frac{6}{7\cdot 8^k}\right) - 1.
\end{align*}
Also, since
$$
\sum_{\substack{F=\Q(\sqrt{r}) \\ |\cD_F|\ge X}} \frac{1}{|\cD_F|^{k}}
   \;\ll\;
   X^{1-k},
$$
we deduce that
that
$$
\cS(X)
   = \lambda 
   X^{1/2}
   \;+\;
   O\!\left(X^{5/12}\right).
$$
  \end{proof}

By Proposition \ref{most-S_3-C_2} and \ref{Count-S_3xC_2}, we conclude that almost all sextic extensions $L/\Q$ with $\Gal(\widetilde{L}/\Q)\simeq S_3\times C_2$ are the compositum $L=K F$ of a cubic field $K$ and a quadratic field $F$ satisfying $D_F\nmid D_K$. In order to compute the genus statistics over $S_3\times C_2$-fields, we first count cubic fields with a given genus number with an additional divisibility condition on their discriminant, following the method in \cite{cubicwithgenusnumber1}.

\subsection{Counting genus one cubic fields with discriminant not divisible by $r$}
Recall that the discriminant of a cubic field has the form
\[
\cD_K \in \{\, df^2, \; 9df^2, \; 81df^2 \,\},
\]
with $f$ squarefree and $(f,3)=1$. We call $d$ the fundamental discriminant of $K$. In fact, the totally ramified primes are precisely those dividing $f$.  Throughout the paper, we set 
$$
    \nu_p :=\frac{p+1}{1+p+p^2}.
$$

\medskip
For a squarefree number $r$, let $\cG^r$ denote the set of cubic fields $K$ such that $r\nmid \cD_K$ and for each squarefree $f$ coprime to $3$, let $\cG^r(f;X)$ denote the number of fields $K \in \cG^r$ with $|\cD_K| \leq X$ that are totally ramified at the primes dividing $f$, and at no other primes except possibly $3$.  We obtain the following asymptotic formula, which parallels \cite[Proposition 3.1]{cubicwithgenusnumber1} and is proved similarly using local-density computations and the inclusion-exclusion principle.
 
\begin{proposition}\label{cubicfieds:totally_ramified_without3}
Let $r$ and $f$ be squarefree integers such that $3\nmid f$. Let $\cG^r(f; X)$ denote the number of cubic fields $K \in \cG^r$ with discriminant $|\cD_K| \leq X$ that are totally ramified precisely at the primes dividing $f$. Then
\[
\cG^r(f; X) \;=\; \frac{13}{36 \zeta(2)} \prod_{p \mid f} \frac{1}{p(p+1)} 
\left( 1 - \prod_{p \mid r} \nu_p \right) X \;+\; O\!\left( f^{-1} X^{\tfrac{16}{17} + \varepsilon} \right).
\]
\end{proposition}

\begin{proof}
Denote by $M^r(f; X)$  the number of cubic fields that are totally ramified at all primes dividing $f$ and unramified at \emph{at least one} prime dividing $r$. By the inclusion–exclusion principle, we can write
\begin{equation}\label{eqn-1}
\cG^r(f; X) \;=\; \sum_{\substack{(s,\,3f)=1}} \mu(s)\, M^r(sf; X).
\end{equation}
Invoking Taniguchi-Thorne's result \cite{countingcubicfields} on counting cubic fields with specified local ramification conditions, for any squarefree $g$ coprime to $3$, one has
	\begin{equation}\label{eqn-3}
	M^r(g;X)
	\;=\; C^r(g)\,X \;+\; D^r(g)\,X^{5/6} 
	\;+\; O\!\big(g^{16/9} X^{7/9+\varepsilon}\big),
	\end{equation}
	where
	\[
	C^r(g) \;=\; \frac{13}{36\,\zeta(2)}
	\Bigg(\prod_{p\mid g} \frac{1}{p^2+p+1}\Bigg)
	\Bigg(1 - \prod_{p\mid r}\nu_p\Bigg) 
	\,\,\,\text{and}\,\,\, D^r(g)=O(g^{-2}).
	\]
Using \eqref{eqn-3} in \eqref{eqn-1} truncated at $s\leq S$, we obtain
\begin{align*}
    	\cG^{r}(f; X)
	= X \sum_{\substack{(s,3f)=1 \\ s\le S}} \mu(s)\,C^r(sf)
	+ O\!\big(f^{-2}X^{5/6}\big)
	+ O\!\big(f^{16/9}S^{25/9}X^{7/9+\varepsilon}\big)
	+ O\!\big(X f^{-2}S^{-1+\varepsilon}\big).	
\end{align*}
The main term can be written as
	\begin{align*}
		\sum_{\substack{(s,3f)=1 \\ s\le S}} \mu(s)\,C^r(sf)
		&= \frac{13}{36\,\zeta(2)}\Bigg(\prod_{p\mid f}\frac{1}{p^2+p+1}\Bigg)\Bigg(1-\prod_{p\mid r}\nu_p\Bigg)
		\sum_{\substack{(s,3f)=1 \\ s\le S}} \mu(s)\prod_{p\mid s}\frac{1}{p^2+p+1}.
	\end{align*}
Letting $S\to\infty$, the inner sum above converges to
	\[
	\prod_{p\nmid 3f}\left(1-\frac{1}{p^2+p+1}\right).
	\]
Note that for primes $p\mid f$
	\[
	\frac{1}{p^2+p+1}\left(1-\frac{1}{p^2+p+1}\right)^{-1}
	= \frac{1}{p(p+1)}.
	\]
Hence, the Euler product evaluates to
	\[
	\frac{13}{36\,\zeta(2)}\Bigg(\prod_{p\mid f}\frac{1}{p(p+1)}\Bigg)\Bigg(1-\prod_{p\mid r}\nu_p\Bigg).
	\]
Now, choosing $S=f^{-1}X^{1/17}$, we conclude that
	\[
	\cG^{r}(f; X) \;=\; \frac{13}{36\,\zeta(2)} 
	\Bigg(\prod_{p\mid f} \frac{1}{p(p+1)}\Bigg)
	\Bigg(1 - \prod_{p\mid r} \nu_p\Bigg) X
	\;+\; O\!\big(f^{-1}X^{16/17+\varepsilon}\big).
	\]
\end{proof}
In a similar spirit, we also have the following.

\begin{proposition}\label{cubicfieds:totally_ramified_with3}
Let $r$ and $f$ be squarefree integers such that $3\nmid f$.
Let $\cG_1^{r}(f; X)$ denote the number of cubic fields $K \in \cG^r$ with $|\cD_K| \leq X$ that are totally ramified at $3$ as well as at every prime dividing $f$, whose fundamental discriminant $d_K$ satisfies $d_K \equiv 1 \pmod{3}$. We further suppose that $r \nmid \cD_K$. Then
\[
\cG_1^r(f; X)) \;=\; \frac{1}{324 \zeta(2)} 
\prod_{p \mid f} \frac{1}{p(p+1)} 
\!\left(1 - \prod_{p \mid r} \nu_p \right) X 
\;+\; O\!\left(f^{-1} X^{\tfrac{16}{17} + \varepsilon}\right).
\]

\end{proposition}

\begin{proof}
    	Let $M_1^r(g;X)$ be the number of cubic fields $K$ with $|\cD_K|\le X$ and $d_K\equiv 1 \bmod 3$ which are totally ramified at $3$ as well as at every prime $p\mid g$  and unramified at least one prime dividing $r$. Then, by inclusion-exclusion, we have
	\begin{equation*}	\cG_1^r(f;X)=\sum_{\substack{(s,3f)=1}}\mu(s)\,M_1^r(sf;X).
	\end{equation*}
Again, counting cubic fields, as in Taniguchi-Thorne  \cite{countingcubicfields}, we deduce that.
	\begin{equation*}
	M^r_1(g;X)=C_1^r(g)\,X+\; D_1^r(g)\,X^{5/6} 
	\;+\; O\!\big(g^{16/9} X^{7/9+\varepsilon}\big),
	\end{equation*}
	with the main coefficient
	\begin{equation*}
	C_1^r(g)=\frac{1}{117}C^r(g)\,\,\text{and}\,\,  D_1^r(g)=O(g^{-2}).
	\end{equation*}
 
The leading factor $1/117$ above arises entirely from the local factor at $p=3$. 
In this case, the normalised $3$--adic contribution produces the factor 
$13/36$. When we require total ramification at $3$ together with 
$d\equiv1\pmod{3}$, the admissible $3$--adic orbits of binary cubic forms shrink substantially.  As in \cite[Proposition~3.2]{cubicwithgenusnumber1}, one can compute the Haar measure of this set and multiplying by the normalizing 
factors $(1-3^{-2})(1-3^{-3})$ obtains $\frac{1}{117}$. Multiplying this by the main term and bounding the error terms as in Proposition \ref{cubicfieds:totally_ramified_without3}, we get the desired result.    
\end{proof}

For $r$ squarefree, let $\cH^r \subseteq \cG^r$ be the subset of those fields with genus number one and let $\cH^r(X)$ denote the number of cubic fields $K$ in $\cH^r$ with $|\cD_K| \leq X$.
\begin{proposition}\label{cubic:genusone:local}
    Let $r$ be a squarefree positive integer.  Then
\[
\cH^r(X) \;=\; \left(1 - \prod_{p \mid r} \nu_p \right) \cdot \frac{29}{81 \zeta(2)} 
\prod_{p \equiv 2 \bmod{3}} \left( 1 + \frac{1}{p(p+1)} \right) X \;+\; O\!\left(X^{16/17+\epsilon}\right).
\]
\end{proposition}

\begin{proof}[Proof of Proposition \ref{cubic:genusone:local}]
By Theorem \ref{genus:quad-cubic}, we have the following description for $\cH^r$.  
Define
\begin{align*}
	\cH_1^r & := \{\, K \in \cG^r : p \equiv 2 \pmod{3} \ \text{for all } p \mid f \,\}\\[4pt]
	\cH_2^r & := \{\, K \in \cH_1^r : 3 \ \text{is totally ramified and } d \equiv 1 \pmod{3} \,\}.
\end{align*}
Then
\begin{equation*}
\cH^r = \cH_1^r \setminus \cH_2^q.
\end{equation*}
For $X > 0$, write
\begin{align*}
    \cH_i^r(X) := \{\, K \in \cH_i^r : |\cD_K| \leq X \,\} \,\,  \text{for } i=1,2.
\end{align*}
It follows that
\[
\cH^r(X) = \cH_1^r(X) - \cH_2^r(X).
\]
If a prime $p \neq 2,3$ is totally ramified in a cubic field, then (see \cite[Example~6.10]{Ishidabook})
\[
\left(\tfrac{d}{p}\right) = 1 \quad \iff \quad p \equiv 1 \pmod{3}.
\]
Thus, only totally ramified primes $p \equiv 1 \pmod{3}$ contribute to the genus number. 
For the prime $3$, the analogous condition
\[
\left(\tfrac{d}{3}\right) = 1 \quad \iff \quad d \equiv 1 \pmod{3},
\]
singles out an exceptional subfamily $\cH_2^r$. Let $T$ denote the collection of all squarefree integers $f$ such that $p \equiv 2 \pmod{3}$ for every prime $p \mid f$.  
Then
\[
\cH_1^r(X) \;=\; \sum_{f \in T} \cG^r(f;X).
\]
Therefore,
\[
\cH_1^r(X) \;=\; \!\left(1 - \prod_{p \mid r} \nu_p \right) \frac{13}{36 \zeta(2)} 
\prod_{p \equiv 2 \bmod{3}} \left( 1 + \frac{1}{p(p+1)} \right) X \;+\; O\!\left(X^{16/17+\epsilon}\right).
\]
For $\cH_2^r(X)$, by Theorem \ref{cubicfieds:totally_ramified_with3} 
\begin{align*}
    \cH_2^r(X) \;=\; \!\left(1 - \prod_{p \mid r} \nu_p \right) \frac{1}{324 \zeta(2)} 
\prod_{p \equiv 2 \bmod{3}} \left( 1 + \frac{1}{p(p+1)} \right) X \;+\; O\!\left(X^{16/17+\epsilon}\right).
\end{align*}
Subtracting the terms above, we conclude that
\begin{align*}
    \cH^r(X) \;=\; \!\left(1 - \prod_{p \mid r} \nu_p \right) \frac{29}{81 \zeta(2)} 
\prod_{p \equiv 2 \bmod{3}} \left( 1 + \frac{1}{p(p+1)} \right) X \;+\; O\!\left(X^{16/17+\epsilon}\right).
\end{align*}
\end{proof}

\subsection{Counting genus one $S_3\times C_2$ fields}
We now count $S_3\times C_2$-fields of genus one, which illustrates the main ideas of the general result in a simpler setting, while yielding a cleaner final count. Since $\cS$ consists of almost all $S_3\times C_2$-fields, we denote by $\cF$ the subset of $\cS$ consisting of fields with genus number one. For any $K(\sqrt{d}) \in \cS$, by \Cref{genusnumber:sixdegree},  
we have $K(\sqrt{d}) \in \cF$ if and only if $g_K = 1$ and $g_F = 1$, where $F = \Q(\sqrt{d})$.
Recall that the quadratic fields with genus number one are precisely 
\[
\Q\!\left(\sqrt{p}\right), \quad p \equiv 1 \pmod{4}, 
\qquad\text{and}\qquad 
\Q\!\left(\sqrt{-p}\right), \quad p \equiv 3 \pmod{4},
\]
where $p$ denotes a rational prime. Let $\cF^+$ and $\cF^-$ be the subsets of $\cF$ consisting of sextic fields with positive 
and negative discriminants, respectively. Then
\begin{align*}
    \cF^{+} &= \left\{\, L = KF \;\middle|\; [K:\Q]=3,\ [F:\Q]=2,\ g_K=1,\ g_F=1,\ \operatorname{sgn}(\cD_F)=+1 \,\right\}, \\
    \cF^{-} &= \left\{\, L = KF \;\middle|\; [K:\Q]=3,\ [F:\Q]=2,\ g_K=1,\ g_F=1,\ \operatorname{sgn}(\cD_F)=-1 \,\right\}.
\end{align*}
By Equation \eqref{compositum:disc} and \Cref{disjointness}, 
the discriminant of such a sextic extension $L = KF$ is
\[
\cD_L = \cD_K^2 \cD_F^3.
\]
Define
\[
\cF^{\pm}(X) = \#\{\, L \in \cF^{\pm} : |\cD_L| \leq X \,\}, \quad \cF(X) = \#\{\, L \in \cF : |\cD_L| \leq X \,\},   
\]
where the sign of $\cD_L$ is determined entirely by $\operatorname{sgn}(\cD_F)$.
Also observe that By definition $\cF(X) = \cQ_{0,0}(X)$.
Since all quadratic fields with genus number one and positive discriminants are of the form 
$\Q(\sqrt{q})$, where $q$ is prime and $q \equiv 1 \pmod{4}$, we obtain

\begin{align*}
	\cQ_{0,0}(X)&=\cF(X) 
	=\cF^+(X) + \cF^-(X)\\
    &=\sum_{\substack{q \leq X^{1/3} \\ q \equiv 1 \bmod{4}\\ q \text{ prime}}}  
	   \cH^q\!\left( \sqrt{ \frac{X}{q^3}} \right) + \sum_{\substack{q \leq X^{1/3} \\ q \equiv 3 \bmod{4}\\q \text{ prime}}} 
	   \cH^q\!\left( \sqrt{ \frac{X}{q^3}} \right) \\
	&= \frac{29}{81\zeta(2)} \sum_{\substack{q \leq X^{1/3} \\q \text{ prime}}} 
	    (1-\nu_q) 
	    \prod_{p \equiv 2 \bmod{3}} \!\left(1 + \frac{1}{p(p+1)}\right) 
	    \left(\frac{X}{q^3}\right)^{1/2} 
	    + O\!\left(\left(\frac{X}{q^3}\right)^{8/17 + \epsilon/2}\right) \\
	&= \frac{29}{81\zeta(2)} \prod_{p \equiv 2 \bmod{3}} \!\left(1 + \frac{1}{p(p+1)}\right) 
	    \sum_{\substack{q \leq X^{1/3} \\ q \text{ prime}}} 
	    \frac{(1-\nu_q)}{q^{3/2}}\,\, X^{1/2} 
	    + O\!\left(X^{8/17 + \epsilon}\right)\\
    &=\delta \, X^{1/2} +  O\!\left({X}^{8/17 + \epsilon}\right),
\end{align*}
where
$$
    \delta := \frac{29}{81\zeta(2)} \prod_{p \equiv 2 \bmod{3}} \!\left(1 + \frac{1}{p(p+1)}\right) 
	    \sum_{ q \text{ prime}} 
	    \frac{q^{1/2}}{q^2 + q + 1}.
$$

  \medskip

 \subsection{Counting $S_3\times C_2$ fields with genus number $2^{k}3^{\ell}$ }
Recall that for $\ell\geq 0$, $ T_{\ell} $ denotes the set of squarefree integers $ n $ that are coprime to 3 and have exactly $ \ell $ prime factors $ p$ satisfying $ p \equiv 1 \pmod{3} $. Also, set $T_{-1}$ as the empty set.

\begin{lemma}\label{count-genus-number-S_3xC_2-lemma}
Let $\ell\geq 0$ and $r \geq 1$ be a squarefree integer and let $\cF_\ell^{\,r}(X)$ denote the number of cubic fields $K$ such that $g_K = 3^{\ell}$, $r \nmid \cD_K$ and $|\cD_K| \leq X$.
Then
\begin{align*}
    \cF_\ell^{\,r}(X) \;=\; \frac{1}{\zeta(2)} 
&\left[ \frac{29}{27} \sum_{f \in T_{\ell}} \prod_{p \mid f} \frac{1}{p(p+1)}
+ \frac{1}{108} \sum_{f \in T_{\ell-1}} \prod_{p \mid f} \frac{1}{p(p+1)} \right]
\left(1 - \prod_{p \mid r} \nu_p \right) X \\
&\;+\; O\!\left(X^{16/17 + \varepsilon}\right),
\end{align*}
\end{lemma}
\noindent
This provides an asymptotic formula for the number of cubic fields with prescribed genus number and restricted ramification. In particular, it refines the result of \cite[Section 5]{countingcubicfields}, which counts cubic fields with discriminant of the form   $\cD_K \;\in\; \{\, d f^2,\; 9d f^2,\; 81d f^2 \,\}$. where $f$ is squarefree and $3\nmid f$, with specified ramification at the primes dividing $f$. Our refinement further stratifies this count by the exact genus number $g_K = 3^{\ell}$, determined by the number of totally ramified primes $p \equiv 1 \pmod{3}$.

\begin{proof}[Proof of Lemma \ref{count-genus-number-S_3xC_2-lemma}]
Consider the following subsets of $ \cG^r $.
\[
\begin{aligned}
\mathcal{F}_{\ell}^{(1)} &= \left\{ K \in \cG^r : f \in T_{\ell} \right\}, \\
\mathcal{F}_{\ell}^{(2)} &= \left\{ K \in \cG^r : f \in T_{\ell},\, 3 \text{ totally ramified, and } d \equiv 1 \pmod{3} \right\}, \\
\mathcal{F}_{\ell}^{(3)} &= \left\{ K \in \cG^r : f \in T_{\ell-1},\, 3 \text{ totally ramified, and } d \equiv 1 \pmod{3} \right\}.
\end{aligned}
\]
By Fr\"ohlich's Theorem \ref{genus:quad-cubic}, we have the decomposition
\[
\mathcal{F}_{\ell} = \left(\mathcal{F}_{\ell}^{(1)} \setminus \mathcal{F}_{\ell}^{(2)}\right) \cup \mathcal{F}_{\ell}^{(3)}.
\]
Hence, the counting function $ \cF_{\ell}^r(X) $ satisfies

\begin{equation}\label{eqn-2}
\cF_{\ell}^r(X) = \cF_{\ell}^{r,1}(X) - \cF_{\ell}^{r,2}(X) + \cF_{\ell}^{r,3}(X),
\end{equation}
where 
\[
\cF_{\ell}^{r,i}(X) = \#\{ K \in \mathcal{F}_{\ell}^{(i)} : |\Delta| \leq X \} \quad \text{for } i = 1,2,3.
\]
Write
\[
\cF_{\ell}^{r,i}(X) 
= \sum_{f \in T_{n_i}} G_i(f; X),
\]
where
\[
G_i(f; X) \;=\;
\begin{cases}
    \cG^{r}(f; X), & i = 1, \\[6pt]
    \cG_1^{r}(f; X), & i = 2,3.
\end{cases}
\]
Applying the asymptotic formulas for $\cG^{r}(f; X)$ and $\cG_1^{r}(f; X)$ as in Proposition \ref{cubic:genusone:local} and \ref{cubicfieds:totally_ramified_without3} and summing over $f \in T_{n_i}$, we have
\[
\cF_{\ell}^{r,i}(X) \;=\; \frac{c_i}{\zeta(2)} \, X 
\sum_{f \in T_{n_i}} \prod_{p \mid f} \frac{1}{p(p+1)} 
\left(1 - \prod_{p \mid r} \nu_p\right) 
\;+\; O\!\left(X^{\tfrac{16}{17}+\varepsilon}\right),
\]
where
\[
c_i \;=\; 
\begin{cases}
    \dfrac{13}{36}, & i = 1, \\[6pt]
    \dfrac{1}{324}, & i = 2,3,
\end{cases}
\qquad
n_i \;=\;
\begin{cases}
    \ell, & i = 1,2, \\[6pt]
    \ell - 1, & i = 3.
\end{cases}
\]
Using this in Equation \eqref{eqn-2}, we deduce that
\[
\cF_{\ell}^r(X) = \frac{1}{\zeta(2)} \left[ \frac{29}{81} \sum_{f \in T_{\ell}} \prod_{p \mid f} \frac{1}{p(p+1)} + \frac{1}{324} \sum_{f \in T_{\ell-1}} \prod_{p \mid f} \frac{1}{p(p+1)} \right] X + O\left(X^{\frac{16}{17} + \varepsilon}\right),
\]
which proves \Cref{count-genus-number-S_3xC_2-lemma}.
\end{proof}

\noindent
We now count the number of fields $L \in \mathcal{S}$ whose genus number is of the form  $g_L = 2^{k} 3^{\ell}$. For integers $k,\ell\geq 0$, recall that
\[
\cQ(k,\ell):=\#\{\, L \in \mathcal{S} : g_L =2^{k} 3^{\ell} \,\}.
\]
Let $\omega(n)$ denote the number of distinct prime divisors of $n$.
\begin{proof}[Proof of Theorem \ref{S_3-C_2-theorem}]
Note that $\cQ(k,\ell)$ can be expressed as
\[
\cQ(k,\ell)
= \sum_{\substack{ F= \Q(\sqrt{D}), \\  |\cD_F|\leq X^{1/3}, \\\omega(\cD_F) = k+1}}
\cF_{\ell}^{\cD_F}\!\left(\sqrt{\tfrac{X}{\cD_F^3}}\right).
\]
Applying Lemma \ref{count-genus-number-S_3xC_2-lemma}, we have
$$
\cQ(k,\ell) = \sum_{\substack{ F= \Q(\sqrt{D}), \\  |\cD_F|\leq X^{1/3}, \\\omega(\cD_F) = k+1}}\frac{B(l)}{\zeta(2)} \, \left(1 -  \prod_{p \mid \cD_F} \nu_p\right) \sqrt{\frac{X}{\cD_F^3}} + O\left(\sum_{\substack{ F= \Q(\sqrt{D}), \\  |\cD_F|\leq X^{1/3}, \\\omega(\cD_F) = k+1}}
	\left(\sqrt{\tfrac{X}{\cD_F^3}}\right)^{16/17+\varepsilon}\right),
$$
The error term 
\begin{align*}
 \sum_{\substack{ F= \Q(\sqrt{D}), \\  |\cD_F|\leq X^{1/3}, \\\omega(\cD_F) = k+1}} 
	\left(\sqrt{\tfrac{X}{\cD_F^3}}\right)^{16/17+\varepsilon} \ll X^{8/17+\varepsilon} \sum_{r} \frac{1}{r^{24/17}} \ll X^{8/17+\varepsilon}.
\end{align*}
For the main term, the tail can be estimated as
$$
 \sum_{\substack{ F= \Q(\sqrt{D}), \\  |\cD_F|\leq X^{1/3}, \\\omega(\cD_F) = k+1}} 
 \frac{B(l)}{\zeta(2)} \, \left(1 -  \prod_{p \mid r} \nu_p\right) \sqrt{\frac{X}{\cD_F^3}} \ll  \sqrt{X}\sum_{r>X^{1/3}}\frac{1}{r^{3/2}} \ll X^{1/2 - 1/6} = X^{1/3}. 
$$
Putting everything together, we deduce the theorem.
\end{proof}

\medskip

\section{\bf Genus statistics for $S_3\times C_q$-fields}\label{S_3xC_q}
\medskip
Let $K$ be a cubic field and $F$ be a cyclic field of prime degree $q\ge 5$.  
Since $\gcd([K:\mathbb{Q}], [F:\mathbb{Q}]) = 1$, by Lemma
\ref{compositum:genusfields} the genus field of the compositum satisfies
\[
    (KF)^{*} = K^{*} F^{*}.
\]
Combined with Theorem \ref{genus:quad-cubic} and Lemma \ref{lemma:genus:number:degree:p:extensions}, this yields the following description of the genus number of $L = KF$.

\begin{lemma}\label{genusnumber:S3Cq}
Let $L=KF$, with $K$ a cubic field and $F$ a cyclic extension of prime degree $q\ge 5$. Then the genus number of $L$ is
\begin{equation*}
    g_L =
    \begin{cases}
        3^{s}\, q^{\Psi_q(\mathcal{D}_F)-1}   & \text{if $K$ is noncyclic over $\mathbb{Q}$}, \\[4pt]
        3^{s-1}\, q^{\Psi_q(\mathcal{D}_F)-1} & \text{if $K$ is cyclic over $\mathbb{Q}$},
    \end{cases}
\end{equation*}
where $s$ is the number of primes $p$ totally ramified in $K$ for which  $d\equiv 1\pmod{3}$, $d$ is the fundamental discriminant of $K$
and  $\Psi_q(a)$ is the number of primes $p\mid a$ such that $p\equiv 1 \pmod q$ or $p=q$.
\end{lemma}

Let $K/\mathbb{Q}$ be a finite abelian extension with Galois group 
$G$. By the conductor-discriminant formula
$$
    |\mathcal{D}_{K}|
    = \prod_{\chi\in\widehat{G}} f_\chi,
$$
where $f_{\chi}$ denotes the Dirichlet conductor of $\chi$. If $K/\mathbb{Q}$ is cyclic of prime degree $q$, all non-trivial characters have the same conductor $f$ and we have
\[
    |\mathcal{D}_{K}| = f^{\,q-1},
\]
where $f$ is the minimal modulus such that $K\subseteq \Q(\zeta_f)$, and its prime divisors are exactly the primes ramified in $K$.\\

We now consider the family of composita of a cubic $S_3$-field and a cyclic field of prime degree $q$. In \cite[Theorem 1.1]{MasriThorneTsaiWang2020}, Masri, Thorne, Tsai and Wang showed that Malle’s conjecture holds for groups of the form $G \times A$ with $G \in \{S_3, S_4, S_5\}$ and $A$ abelian.  Specializing to $G=S_3$ and $A=C_q$, it follows that as $X\to\infty$, there exists a $c(q)>0$ such that
\begin{equation*}
    |\cM(S_3\times C_q, 3q, X)|\sim c(q) X^{1/q}.
\end{equation*}

In order to analyze the distribution of genus numbers within the family
$\cM(S_3\times C_q, 3q, X)$, we require a corresponding counting result for cubic fields with a prescribed genus number.  Towards this, we invoke \cite[Theorem~1.3]{cubicwithgenusnumber1}. Unlike the original statement, which distinguishes between positive and negative
discriminants, we treat both the cases together. Let $N_3^{\ell}(X)$ denote the number of cubic fields $K$ satisfying $g_K = 3^{\ell}$ and $|\mathcal{D}_K|\le X$. Then
\begin{equation}\label{lemma:count-genus-number-S_3}
    N_3^{\ell}(X)
    = \frac{\alpha_{\ell}}{\zeta(2)}\,X
      + O\!\left(X^{16/17+\varepsilon}\right),
\end{equation}
where 
\begin{align*}
	\alpha_{\ell}=\frac{29}{27} \sum_{f \in T_{\ell}} \prod_{p \mid f} \frac{1}{p(p+1)}
	+ \frac{1}{108} \sum_{f \in T_{\ell-1}} \prod_{p \mid f} \frac{1}{p(p+1)}.
\end{align*}
 
We are now ready to count the number of $L$ in $\cM(S_3\times C_q, 3q, X)$ with a prescribed genus number.
\begin{proof}[Proof of Theorem \ref{S_3-C_q-theorem}]
Since $[K:\mathbb{Q}]=3$ and $[F:\mathbb{Q}]=q$ are coprime, the fields $K$ and
$F$ are linearly disjoint and satisfy
\[
    \mathcal{D}_{KF}
    = \mathcal{D}_{K}^{\,q}\, \mathcal{D}_{F}^{\,3}.
\]
Thus the condition $|\mathcal{D}_{KF}|\le X$ is equivalent to
\[
    |\mathcal{D}_K|^{\,q}\,|\mathcal{D}_F|^{\,3} \le X.
\]
Therefore,
\begin{align*}
    |\cQ^q_{k,l}(X)|
    &= 
    \#\bigl\{\, L \in \cM(S_3\times C_q, 3q, X) : \ g_L = 3^{\ell}q^{k} \,\bigr\} \\
    &=
    \sum_{\substack{
        F:\, \Gal(F/\mathbb{Q})\simeq C_q\\
        g_F = q^{k},\ |\mathcal{D}_F|^{3}\le X}}
    \
    \sum_{\substack{
        K\ \text{cubic},\ \Gal(\widetilde{K}/\mathbb{Q})\simeq S_3\\
        g_K = 3^{\ell},\ |\mathcal{D}_K|^{q}\le X/|\mathcal{D}_F|^{3}}}
    1.
\end{align*}
Using \eqref{lemma:count-genus-number-S_3}, we obtain
\begin{align}\label{eqn-a-1}
    |\cQ^q_{k,l}(X)|
    &= 
    \sum_{\substack{
        F:\, \Gal(F/\mathbb{Q})\simeq C_q\\
        g_F = q^{k},\ |\mathcal{D}_F|\le X^{1/3}}}
    N_3^{\ell}\!\left( \left(\frac{X}{\mathcal{D}_F^3}\right)^{1/q} \right) \nonumber\\[6pt]
    &=
    \sum_{\substack{
        F:\, \Gal(F/\mathbb{Q})\simeq C_q\\
        \Psi_q(\mathcal{D}_F)=k+1,\ |\mathcal{D}_F|\le X^{1/3}}}
    \left[
        \frac{\alpha_\ell}{\zeta(2)}
        \left(\frac{X}{\mathcal{D}_F^3}\right)^{1/q}
        + O\!\left(
            \left(\frac{X}{\mathcal{D}_F^3}\right)^{16/(17q)+\varepsilon}
        \right)
    \right].
\end{align}
We show that the sum
\begin{equation}\label{eqn-sum}
    \sum_{\substack{
        F:\, \Gal(F/\mathbb{Q})\simeq C_q\\
        \Psi_q(\mathcal{D}_F)=k+1}}
    \frac{1}{\mathcal{D}_F^{3/q}}
\end{equation}
is convergent. For each cyclic extension $F/\mathbb{Q}$ of degree $q$, one has
$\mathcal{D}_F = f^{\,q-1}$, where $f$ is the conductor of $F$. Let $N_f(q)$ be the number of cyclic extensions of degree $q$ and conductor $f$.
As $N_f(q)$ is bounded by the number of order $q$ subgroups of
$(\mathbb{Z}/f\mathbb{Z})^\times$. If $q^{r}$ is the largest power of $q$ dividing $\varphi(f)$, then one has 
\[
    N_f(q)
    \le \frac{q^{r}-1}{q-1}
    \le \varphi(f)
    \le f.
\]
Thus
\begin{equation*}
\sum_{\substack{
        F:\, \Gal(F/\mathbb{Q})\simeq C_q \\ \Psi_q(\mathcal{D}_F)=k+1}
        }
\frac{1}{\mathcal{D}_F^{3/q}}
\ll
\sum_{\substack{\omega(f)=k+1}}
\frac{N_f(q)}{f^{3(q-1)/q}}
\ll
\sum_{\substack{\omega(f)=k+1}}
\frac{1}{f^{2-\frac{3}{q}}} \ll 1. 
\end{equation*}
Therefore, from Equation \eqref{eqn-a-1}, we conclude that
\begin{equation*}
    |\cQ^q_{k,l}(X)| = C_q(k,\ell) X^{1/q} + O(X^{16/17q}),
\end{equation*}
where
\begin{equation}\label{implied-constant-1}
    C_q(k,\ell)
    =
    \frac{\alpha_{\ell}}{\zeta(2)}
    \sum_{\substack{
        F:\, \Gal(F/\mathbb{Q})\simeq C_q\\
        \Psi_q(\mathcal{D}_F)=k+1}}
    \frac{1}{\mathcal{D}_F^{3/q}}.
\end{equation}
\end{proof}

\section{\bf Genus statistics for $D_4-$fields}\label{D_4}
\medskip

In this section, we study the distribution of genus numbers in the family of $D_4$-fields. A quartic field containing a quadratic subfield must be a $D_4$-, $C_4$-, or $V_4$-field. The number of $C_4$-fields with discriminant up to $X$ is $O(\sqrt{X})$, and that of $V_4$-fields is $O\bigl(\sqrt{X}(\log X)^2\bigr)$, whereas the number of $D_4$-fields is $\gg X$. Thus, almost all such quartic fields are $D_4$-fields.\\

Motivated by the conjectures of Malle~\cite{Mal04} and Bhargava~\cite{Bha07}, Cohen, Diaz y Diaz, and Olivier~\cite{CDO02} established an asymptotic formula for the number of $D_4$-fields $K$ with $|\cD_K|\leq X$. This was refined by McGown and Tucker \cite{McTu24}, who proved that
\begin{equation*}
    \cM(D_4, 4, X) =  \gamma\, X + O(X^{5/8+\epsilon}), 
\end{equation*}
where
\begin{equation*}
    \gamma:=\frac12
      \sum_{\substack{ [k:\mathbb{Q}]=2}}
      \frac{1}{2^{r_2(k)}\,\mathcal{D}_k^{\,2}}
      \cdot
      \frac{\zeta_k^{*}(1)}{\zeta_k(2)}.
\end{equation*}
Here $r_2(k)$ is the number of complex embeddings of $k$, $\zeta_k(s)$ its Dedekind zeta function and $\zeta_k^{*}(1)$ the
residue of $\zeta_K(s)$ at $s=1$. Their proof relies on a parametrization of quadratic extensions over an arbitrary number field due to Cohen, Diaz y Diaz, and Olivier \cite{CDO02}, which we also employ towards our application.\\

For a number field $k$, define
	$$
	V(k) := \left\{ u \in k^* : (u) = \mathfrak{q}^2 \text{ for some ideal } \mathfrak{q} \subset \mathcal{O}_k \right\}
	$$
	and the $2$-Selmer group of $k$ as
	\[
	S(k) := V(k) / (k^*)^2.
	\]
	Let $ A(k) $ denote the set of integral, squarefree ideals $ \mathfrak{a} \subset \mathcal{O}_k $ with $ \mathfrak{a} \in \text{Cl}(k)^2$, a square element in the class group. The \textit{parametrization theorem} \cite{CDO02} gives a natural bijection
	\[
	A(k) \times S(k) \longrightarrow \left\{ \text{Quadratic extensions } K/k \right\} \cup \{k\}.
	\]
	Each pair $ (\mathfrak{a}, u) \in A(k) \times S(k) $ corresponds to a (possibly trivial) quadratic extension $ K $ of $ k $. Given a pair $ (\mathfrak{a}, u) $, one can compute the relative discriminant $ \cD_{K/k} $ as follows. Choose an ideal $ \mathfrak{q} \subset \mathcal{O}_k $ such that
	\[
	\mathfrak{a} \mathfrak{q}^2 = (\alpha_0)\,\, \text{ and} \,\,(\mathfrak{q}, (2)) = 1
	\]
	for some $ \alpha_0 \in k $. Let $\mathfrak{c}\mid 2$ be the largest ideal coprime to $\mathfrak{a}$ such that $x^2 \equiv \alpha_0 u \pmod{\mathfrak{c}^2}$ is solvable in $ (\mathcal{O}_k / \mathfrak{c}^2)^\times $. Then the relative discriminant is 
	\[
	\cD_{K/k} = \frac{4 \mathfrak{a}}{\mathfrak{c}^2}.
	\]
	This parametrization enables us to count quadratic extensions over a number field with local conditions on the discriminant, which is key to counting $D_4$-fields with a given genus number.
\medskip
\subsection{Counting $D_4$-fields with a prescribed genus number }
Let $K/\Q$ be a quartic field whose Galois closure has Galois group $D_4$ and let $k:=\Q(\sqrt{d})$ denote its unique quadratic subfield. Let $q$ be a rational prime, which factorizes in $\cO_K$ as
$$
q\cO_K=\fq_1^{\,e_1}\cdots \fq_f^{\,e_f},
$$
with $\fq_1,\dots,\fq_f$ distinct prime ideals of $\cO_K$. Define
$$
e(q):=\gcd(e_1,\dots,e_f).
$$
Since $K/\Q$ has degree $4$ and its Galois closure has Galois group $D_4$, one always has $e(q)\in\{1,2,4\}$. Let $p\mid \cD_K$ be a prime. If
$p\mid\mid \cD_K$, then $e(p)=1$. On the other hand,
\[
p^2\mid \cD_K
\quad\Longleftrightarrow\quad
e(p)\in\{2,4\}.
\]

\noindent
The genus number $g_K$ of a $D_4$–quartic field is determined by this order of ramification at primes dividing $\cD_K$. More precisely, the primes $p$ such that $p^2 \mid \mathcal{D}_K$ are precisely those that contribute to the genus number. To quantify the occurrence of such primes, define 
$\omega^{(2)}(n)$ to be the number of primes $p$ for which $p^2 \mid n$. 

\begin{lemma}\label{lem:omega2-genus}
Let $K$ be a quartic $D_4$-field with unique quadratic subfield $k=$. If $g_K=2^{a}$, then
\[
\Big\lfloor\frac{a}{2}\Big\rfloor
\le \omega^{(2)}(\mathcal{D}_K)
\le a+2.
\]
\end{lemma}

\begin{proof}
By Lemma \ref{lemma:k1k2decomp}, the genus field decomposes as
$k^{*}=k_1^{*}k_2^{*}$, and
$$
g_K
=\frac{[k^{*}:\Q]}{[K_0:\Q]}
=\frac{[k_1^{*}:\Q]\,[k_2^{*}:\Q]}{2}.
$$
If $p^{2}\mid\cD_K$, then $e(p)\in\{2,4\}$ and hence
$\gcd(e(p),p-1)\in\{2,4\}$. In this case, $p$ contributes a factor
$2^{\varepsilon_p}$ to $[k_1^{*}:\Q]$, where $\varepsilon_p\in\{1,2\}$. If $2$ ramifies, the wild part contributes a factor $2^{\delta}$ to $[k_2^{*}:\Q]$, where $\delta\in\{0,1,2\}$. Let $r=\omega^{(2)}(\cD_K)$, and write
$$
g_K = 2^{\,r_1-1}\,4^{\,r_2},
\qquad
r_1,r_2\ge0,\qquad
r_1+r_2\in\{\,r,\ r-1\,\},
$$
where $r_1$ (resp.\ $r_2$) counts the primes contributing $2$ (resp.\ $4$).
Since $g_K=2^{a}$, we have
$ a = r_1 + 2r_2 - 1$. From $r_1+r_2 \ge r-1$, it follows that 
\[ a = r_1+2r_2 -1 \ge (r-2)+r_2 \ge r-2, \] 
giving $r \le a+2$. For the lower bound, note that $r_1+r_2 \ge \lfloor a/2\rfloor$ and that $r = (r_1+r_2) + \eta$, with $\eta\in\{0,1\}$. So $ r\ge \Bigl\lfloor \frac{a}{2} \Bigr\rfloor$.  Hence \[ \left\lfloor \frac{a}{2} \right\rfloor \;\le\; r \;\le\; a+2, \] as claimed. 
\end{proof}

\noindent
Any prime $p \mid \mathcal{D}_K$ with $e(p)=1$ necessarily splits in $k$, 
since ramification or inertness of $p$ in $k$ would force $e(p)\geq 2$. For a fixed nonsquare integer $d$, set
\[
\mathcal{P}(d) := \{\, p \text{ prime} : p \text{ splits in } \mathbb{Q}(\sqrt{d})\,\},
\]
and let $S_d$ be the set of prime ideals of $\cO_k$ lying above primes in
$\cP(d)$:
$$
S_d:=\{\fp\subset\cO_F \;\text{prime} \;:\ \fp\mid p\ \text{for some }p\in\cP(d)\}.
$$

Suppose now that $g_K = 1$. By Theorem~\ref{genus:quad-cubic}, this forces 
$g_k = 1$, which holds if and only if $k = \mathbb{Q}(\sqrt{d})$ where $d$ 
is a prime discriminant of the form $\pm q$ with $q$ a rational prime $1\pmod 4$. Then  Lemma~\ref{lem:omega2-genus}  implies that 
$p^2 \nmid \mathcal{D}_K$ for every prime $p \neq q$, so $e(p) = 1$ for all 
such $p \mid \mathcal{D}_K$. Consequently, every prime divisor $p \neq q$ of 
$\mathcal{D}_K$ lies in $\mathcal{P}(d)$, the relative discriminant 
$\mathcal{D}_{K/k}$ is supported on $S_d$, and $\mathrm{N}\mathcal{D}_K$ 
is squarefree.\\

\noindent
Let
$$
S_d^{(1)} := \{\, \mathfrak{p}_1 : p\mathcal{O}_k = \mathfrak{p}_1\mathfrak{p}_2,
\ p \in \mathcal{P}(d)\,\}.
$$
An ideal $\mathfrak{a} \subset \mathcal{O}_k$ supported on $S_d^{(1)}$ has 
squarefree norm, since at most one prime above each $p \in \mathcal{P}(d)$ 
can divide $\mathfrak{a}$. We are thus led to counting ideals in the set
\[
\mathcal{A}_d := \Bigl\{\, \mathfrak{a} \subset \mathcal{O}_k : 
\mathfrak{a}\ \text{squarefree},\quad
\mathfrak{p} \mid \mathfrak{a} \Rightarrow \mathfrak{p} \in S_d,\quad
\mathrm{N}(\mathfrak{a})\ \text{squarefree}\,\Bigr\}.
\]
Let $\cA_d(X) := \#\{\mathfrak{a} \in \mathcal{A}_d 
: \mathrm{N}(\mathfrak{a}) \leq X\}$ be the associated counting function. We show that the set $\cA_d$ has positive density in the set of ideals.
\begin{lemma}\label{counting_ideals_above_split_primes}
Let $k = \mathbb{Q}(\sqrt{d})$. There exists a constant $C(d) > 0$ such that
\[
\cA_d(X) \sim C(d)\, X \qquad \text{as } X \to \infty.
\]
\end{lemma}

\begin{proof}
    
Consider the Dirichlet series
\[
F(s)
   = \sum_{\substack{\fa\ \text{squarefree}\\ \fp\mid\fa\Rightarrow \fp\in S_d \\ \Norm\fa \ \text{squarefree}} }
     \frac{1}{\Norm(\fa)^s},
   \qquad \Re(s)>1.
\]
For any $p\in\cP_d$, the squarefree norm condition ensures that primes $\fp_1$ and $\fp_2$ cannot both occur in the Euler product above. Hence, for $\Re(s)>1$
\begin{align*}
F(s)&=\prod_{\fp\in S_d^{(1)}}\Bigl(1+\frac{2}{\Norm(\fp)^s}\Bigr) = \prod_{\fp \in S_d^{(1)}} 
\left(
\left(1 + \frac{1}{\Norm \fp^s}\right)^2 
+ O\left(\frac{1}{\Norm \fp^{2s}}\right)\right)\\
& =\prod_{\fp\in S_d}
    \left(1+\frac{1}{\Norm\fp^s}\right)
    \left( 1+O\!\left(\frac{1}{\Norm\fp^{2s}}\right)\right)\\
&=\prod_{\fp\in S_d}
     \left(1-\frac{1}{\Norm\fp^s}\right)^{-1} \left(1-\frac{1}{\Norm\fp^{2s}}\right)\left( 1+O\!\left(\frac{1}{\Norm\fp^{2s}}\right)\right)\\
&= \zeta_k(s) G_k(s) + O(1),
\end{align*}
where 
\begin{equation*}
    G_k(s):= \prod_{\fp\in S_d}  \left(1-\frac{1}{\Norm\fp^{2s}}\right) \prod_{\fp 
    \text{ does not split in }k} \left(1-\frac{1}{\Norm\fp^{s}}\right).
\end{equation*}

The Euler product for $G_K(s)$ converges absolutely for $\Re(s)>1/2$ and hence near $s=1$
\[
F(s)=\frac{\zeta^*_k(1)\,G_k(1)}{s-1} + O(1).
\]

\noindent
By the the classical Tauberian theorem, we conclude that as $X\to \infty$
\[
\cA_d(X)
\;\sim\;
C(d)\, X,
\]
where $C(d):= \zeta^*_k(1)\,G_k(1)$.
\end{proof}

For our purposes, we have to count ideals as above, but with additional arithmetic conditions. For $k=\Q(\sqrt{d})$ as above, let $\Cl_{k}^{(4)}$ denote the ray class group of $k$ with modulus $4\mathcal{O}_k$. For $Y \geq 1$, define
\[
\cA_d^\square(Y)
:=
\Bigl\{
\fa\subset\cO_k \; \Bigm| \;
\ \fa\ \text{squarefree},\ 
\fp\mid\fa\Rightarrow\fp\in S_d,\ 
\Norm(\fa)\le Y,\ 
\Norm(\fa)\ \text{squarefree},\ 
[\fa]\in \left(\Cl_{k}^{(4)} \right)^2
\Bigr\}.
\]

\begin{proposition}\label{prop:Adsquare}
Let $k=\Q(\sqrt d)$ be a quadratic field. Then, as $Y\to\infty$ 
\[
\cA_d^\square(Y)
\;\sim\;
\frac{C(d)}{2^{r_2(\Cl_k^{(4)})}}\,Y,
\]
where $C(d)$ is the constant as in Lemma \ref{counting_ideals_above_split_primes} and $r_2(G)$ is the $2$-rank of $G$.
\end{proposition}

\begin{proof}
Set $G := \mathrm{Cl}_k^{(4)}$ and let $\widehat{G}[2]$ denote the 
$2$-torsion subgroup of the character group of $G$. Every quadratic 
character $\chi \colon G \to \{\pm 1\}$ is trivial on $G^2$, and hence factors 
through the quotient $G/G^2$. This induces a natural isomorphism $\widehat{G}[2] \cong 
\widehat{G/G^2}$. By the orthogonality relations for characters, for any $g\in G$
\[
\mathbf{1}_{G^2}(g)
= \frac{1}{|G/G^2|}
\sum_{\chi \in \widehat{G}[2]} \chi(g)
= \frac{1}{2^{r_2(\mathrm{Cl}_k^{(4)})}}
\sum_{\chi \in \widehat{G}[2]} \chi(g).
\]
Applying this identity with $g=[\mathfrak{a}]$ and summing over 
$\mathfrak{a} \in \mathcal{A}_d(Y)$, we deduce that
\[
\cA_d^\square(Y)
= \frac{1}{2^{r_2(\mathrm{Cl}_k^{(4)})}}
\sum_{\chi \in \widehat{G}[2]}
\sum_{\mathfrak{a} \in \mathcal{A}_d(Y)} \chi(\mathfrak{a}).
\]
\noindent
The trivial character contributes precisely $\cA_d(Y)$ to the inner sum, which satisfies $\cA_d(Y) \sim C(d)\,Y,$  by Lemma~\ref{counting_ideals_above_split_primes}. For each nontrivial $\chi \in \widehat{G}[2]$, by \cite[Lemma 2]{McTu24}, we have
\begin{equation*}
    \sum_{\mathfrak{a} \in \mathcal{A}_d(Y)} \chi(\mathfrak{a}) = O(Y^{1/3}\log Y)
\end{equation*}
and the Proposition follows.
\end{proof}

\noindent
With this asymptotic formula, we return to the arithmetic problem. In parametrizing quadratic extensions of $k$ by pairs $(\fa,u)$, the relative discriminant depends on the local solvability of  $x^{2} \equiv a_{0}u \pmod{\mathfrak{c}^{2}}$ at primes above $2$. For $\mathfrak{c}=(2)$, this condition is determined by the ray class group modulo $4\cO_K$. We state \cite[Lemma 3.5]{CDO02}, specialized to this case.

\begin{proposition}\label{existence_of_sol_for_congurence}
Let $k$ be a number field and let $\fa \subset \cO_k$ be an integral ideal
coprime to $2\cO_k$. Suppose that there exist an ideal $\fq$, also coprime to
$2\cO_k$, and an element $a_0 \in k^\times$ such that
\[
\fa \fq^{2} = a_0 \cO_k .
\]
Then the following conditions are equivalent.
\begin{enumerate}
\item There exists a Selmer class $u \in S(k)$ such that, for some (equivalently
any) lift $\tilde u \in k^\times$ coprime to $2\cO_k$, the congruence
\[
x^{2} \equiv a_0 \tilde u \pmod{4\cO_k}
\]
has a solution $x \in \cO_k$.

\item The ray class of $\fa$ modulo $4\cO_k$ is a square in the ray class group:
\[
[\fa] \in \left( \Cl_k^{(4)} \right)^{2}.
\]
\end{enumerate}
\end{proposition}

Let $\mathfrak{c} \mid 2$ be an ideal of $\mathcal{O}_k$. The 
\emph{ray Selmer group modulo $\mathfrak{c}^2$} is the subgroup 
$S_{\mathfrak{c}^2}(k) \subseteq S(k)$ consisting of classes 
$u \in S(k)$ admitting a lift $\tilde{u} \in k^\times$, 
coprime to $\mathfrak{c}$, such that 
\[
x^2 \equiv \tilde{u} \pmod{\mathfrak{c}^2}
\]
is solvable in $\mathcal{O}_k$. This condition is independent of the 
choice of lift $\tilde{u}$.
In particular, if $\mathfrak{c} = (2) $, then $S_4(k) := S_{(2)^2}(k)$ modulo $4\mathcal{O}_k$.\\

\noindent
In the parametrisation of quadratic extensions $L/k$ by Selmer data 
$(\mathfrak{a}, u)$, the part co-prime to $2$ of the relative discriminant is 
determined by the squarefree ideal $\mathfrak{a}$, while the $2$-part is 
governed by the solvability of $x^2 \equiv a_0 u \pmod{\mathfrak{c}^2}$ 
for $\mathfrak{c} \mid 2$.\\

Specialising to $\mathfrak{c} = (2)$, we deduce that for any squarefree ideal $\mathfrak{a} \subset \mathcal{O}_k$  with  $(\mathfrak{a}, 2\mathcal{O}_k) = 1$, if there exist an ideal $\mathfrak{q}$ with $(\mathfrak{q}, 2\mathcal{O}_k) = 1$ 
and an element $a_0 \in k^\times$ such that $\mathfrak{a}\mathfrak{q}^2 = a_0\mathcal{O}_k$, then
\begin{equation}\label{lem:count_disc_c2}
    \#\{\, L/k\ \text{quadratic} : \mathcal{D}_{L/k} = \mathfrak{a}\,\}
=
\begin{cases}
|S_4(k)| & \text{if } [\mathfrak{a}] \in \bigl(\mathrm{Cl}_k^{(4)}\bigr)^2, \\
0         & \text{otherwise.}
\end{cases}
\end{equation}

\noindent
This reduces the enumeration of quadratic extensions with fixed
relative discriminant to the size of the ray Selmer group modulo $4\cO_k$.
By \cite[Proposition~3.9]{CDO02}, the group
$S_{4}(k)$ is an $\F_2$-vector space with cardinality
\begin{equation*}\label{lem:order-ray-selmer-4}
    |S_{4}(k)|=2^{\,r_u(k)+r_2(\Cl_k^{(4)})-1},
\end{equation*}
where $r_u(k)$ denotes the unit rank of $k$ and $r_2(\Cl_k^{(4)})$ is 2-rank of ray class group $\Cl_k^{(4)}$. This helps us in counting quadratic fields whose relative discriminant is supported on $S_d$ and norm of relative discriminant bounded by $X$.

\begin{proposition}\label{prop:Qd}
Let $k = \mathbb{Q}(\sqrt{d})$ and let $S_d$ denote the set of prime ideals 
of $\mathcal{O}_k$ lying above primes that split in $k$. Define
\[
\mathcal{Q}_d(X)
:=
\left\{
\, L/k \text{ quadratic} \;\middle|\;
\mathrm{Supp}(\mathcal{D}_{L/k}) \subseteq S_d,\ \Norm(\mathcal{D}_{L/k}) \  \text{ is squarefree },\
\Norm(\mathcal{D}_{L/k}) \le X
\right\}.
\]
Then, as $X \to \infty$,
\[
|\mathcal{Q}_d(X)| \sim 2^{\,u(d)-1}\, C(d)\, X,
\]
where $u(d) = 1$ if $d>0$ and $0$ if $d<0$.
\end{proposition}
\begin{proof}
By \cite[Proposition~3.9]{CDO02}, quadratic extensions $L/k$ with 
$\mathrm{Supp}(\mathcal{D}_{L/k}) \subseteq S_d$ and 
$\mathrm{N}(\mathcal{D}_{L/k}) \leq X$ correspond bijectively to pairs 
$(\mathfrak{a}, \alpha)$ with $\mathfrak{a} \in \mathcal{A}_d^\square(X)$ 
and $\alpha \in S_4(k)$, so that
\[
|\mathcal{Q}_d(X)| = |S_4(k)|\, \cA_d^\square(X).
\]
The Proposition now follows from Lemma~\ref{lem:order-ray-selmer-4} and 
Proposition~\ref{prop:Adsquare}.
\end{proof}
\medskip

\subsection{Proof of Theorem~\ref{thm:D4-genus1-upper}}
Let $K$ be a $D_4$-field with $g_K=1$. Let $k\subset K$ be its unique quadratic subfield. Then, we also have $g_k=1$. Thus, $k=\Q(\sqrt{d})$,  where 
$d = \pm q$ for a rational prime $q$ with $d \equiv 1 \pmod{4}$.
Furthermore, for all primes $p\neq q$ dividing $\cD_K$, we have $e(p)=1$. Hence $\cD_{K/k}$ is a squarefree ideal supported on $S_d$ with $\mathrm{N}(\mathcal{D}_{K/F}) \leq 
X/p^2$. Therefore,
\[
\sum_{\substack{
[K:\Q]=4\\
\Gal(\widetilde K/\Q)\cong D_4\\
g_K=1,\ |\cD_K|\le X
}}
1
\;=\;
\sum_{\substack{d=\pm p\\ p\ \text{prime}\\ d\equiv 1 \pmod{4}}}
\left|\mathcal{Q}_d(X)\!\left(\frac{X}{d^2}\right)\right|.
\]
Applying Proposition~\ref{prop:Qd}, we obtain
\[
\sum_{\substack{K\in \cM(D_4,4,X)\\ g_K=1}} 1
\;=\;
\sum_{\substack{d=\pm p\\ p\ \text{prime}\\ d\equiv 1 \pmod{4}}}
 2^{u(d)-1}C(d)\,\frac{X}{d^2}.
\]
Recall that (\cite[Chapter XVI, Lemma 1]{Langbook}) for any $\varepsilon>0$ and any number field $K$, the residue $\zeta_K^*(1) \ll_{\epsilon} \cD_K^{\epsilon}$. Hence, $C(d)\ll_{\varepsilon} d^{\varepsilon}$ for every $\varepsilon>0$ and the series 
\begin{equation}
    \delta_0=\sum_{\substack{d=\pm p\\ p\ \text{prime}\\ d\equiv 1 \pmod{4}}}\frac{2^{u(d)-1}C(d)}{d^{2}}
\end{equation}
converges absolutely. Hence, we conclude that 
$$
\sum_{\substack{K\in \cM(D_4,4,X)\\ g_K=1}} 1
\;\sim\;\delta_0\,X.
$$
This proves the asymptotic for genus one quartic $D_4$-fields. We now obtain a lower bound for quartic $D_4$-fields $K$ with $g_K=2^a$ and $a\geq 1$. Let $K$ be a quartic $D_4$-field and $k=\Q(\sqrt{d})$ be its unique quadratic subfield. Suppose $\omega(\cD_k)=a+1$ and $\cD_{K/k}$ is supported only on primes in $S_d$. Then, we know that
$$
    g_K = 2^a.
$$
Using the parametrization as above and applying Proposition \ref{prop:Qd}, we obtain
$$
    \sum_{\substack{K\in \cM(D_4,4,X)\\ g_K=2^a}} 1 \geq \sum_{\substack{k=\Q(\sqrt{d})\\ \omega(\cD_k)=a+1\\ \cD_k\le X}} \left|\mathcal{Q}_d\left(\frac{X}{\cD_k^2}\right)\right|
\geq\;
\sum_{\substack{k=\Q(\sqrt{d})\\ \omega(d)=a+1\\ \cD_k\le X}}
C(d)\,\frac{X}{\cD_k^{2}}.
$$
Since $C(d)\ll_{\varepsilon}d^{\varepsilon}$ for every $\varepsilon>0$, the series $\sum_{k=\Q(\sqrt{d})} \tfrac{C(d)}{\cD_k^2}$ converges absolutely and we deduce Theorem \ref{thm:D4-genus1-upper}.

\section{\bf Genus statistics for pure quartic fields}\label{pure-impure}
\medskip

In this section, we consider a natural subfamily of $D_4$-fields, namely the pure quartic fields and study their genus statistics. A pure quartic field is of the form $K = \mathbb{Q}(\sqrt[4]{a})$, where $a$ is an integer $\neq \pm 1$, such that is $p^{\nu}\| a$, then $\gcd(\nu, 4)=1$. Let $\cB$ denote the set of all such pure quartic fields. For $K\in \cB$, its discriminant
\begin{align}\label{discrimant_pure_fields}
	\cD_K =
	\begin{cases}
		-4\,a^{3}, & \text{if } a \equiv 1 \pmod{8},\\[2mm]
		-16\,a^{3}, & \text{if } a \equiv 5 \pmod{8},\\[2mm]
		-256\,a^{3}, & \text{if } a \not\equiv 1,5 \pmod{8}.
	\end{cases}
\end{align}
Hence, the only primes that ramify in $K$ are $2$ and those dividing $a$. Define 
\begin{equation*}
    \cB(X):= \left\{
K\in \cB,\,\, |\cD_K|\leq X
\right\}.
\end{equation*}

\noindent
For $K\in \cB$, the genus number can be deduced from  Main theorem of Ishida \cite{Ish81} as
\begin{equation}\label{Theorem:genus_number}
    g_K
= \frac{1}{2} \prod_{p \mid a} (e(p),\, p-1)
\times
\begin{cases}
1, & \text{if } a \equiv 1 \pmod{4},\\[4pt]
2, & \text{if } 2 \mid a,\\[4pt]
2^{\,d-1}, & \text{if } a \equiv 3 \pmod{4},
\ \text{where } d = \min(N,3) \text{ and } 2^{N}\,\|\, (a+1).
\end{cases}
\end{equation}
Let $p$ be an odd prime, and $a=p^{\nu}u$ with $u\in\mathbb{Z}$, $p\nmid u$, and 
$\nu\in\{1,3\}$. Then, the minimal polynomial of $\alpha=\sqrt[4]{a}$ is $f(x)=x^{4}-p^{\nu}u$, which  is irreducible over $\Q$. Indeed, for $\nu=1$, $f$ is an Eisenstein polynomial and if $\nu=3$, $\beta=\alpha/p$ satisfies an Eisenstein polynomial. Hence, $[\mathbb{Q}(\alpha):\mathbb{Q}]=4$, and 
the prime $p$ is totally ramified in $\mathbb{Q}(\alpha)/\mathbb{Q}$. Using this in \eqref{Theorem:genus_number}, we deduce that
\begin{align*}
     g_K
    = 
\begin{cases}
 2^{\omega_3(a)-1} \cdot 4^{\omega_1(a)}, & \text{if } a \equiv 1 \pmod{4},\\[4pt]
 2^{\omega_3(a)} \cdot 4^{\omega_1(a)}, & \text{if } 2 \mid a,\\[4pt]
 2^{\omega_3(a)+d-2} \cdot 4^{\omega_1(a)}, & \text{if } a \equiv 3 \pmod{4},
\ \text{where } d = \min(N,3) \text{ and } 2^{N}\,\|\, (a+1).
\end{cases}
 \end{align*}
		where $ \omega_i(a) $ denotes the number of distinct prime divisors $ p \mid a $ such that $ p \equiv i \pmod{4} $.\\

 In \cite{Ben18}, Benil showed that the number of pure fields of a prime degree $\ell$ with discriminant bounded by $X$ is of the order $X^{1/\ell}\log X$. The quartic case, however, exhibits a different behaviour and we introduce the following parametrisation. Define
\[
\mathcal{A}
:= 
\left\{
n=\prod_{p} p^{a_{p}} : a_{p}\in\{1,3\}
\right\}.
\]
Every pure quartic field $K=\mathbb{Q}(\sqrt[4]{a})$ corresponds uniquely to an integer $a\in\mathcal{A}$, where the exponents are restricted to be 
coprime to $4$. This reduction allows us to count such fields by studying the asymptotic behaviour of 
$\mathcal{A}(X)=\#\{n\in\mathcal{A}:|n|\le X\}$. Partition $\mathcal{A}$ into three disjoint subsets,
\[
\mathcal{A} = \mathcal{A}_1 \sqcup \mathcal{A}_2 \sqcup \mathcal{A}_3,
\]
where
\[
\begin{aligned}
	\mathcal{A}_1 &= \{ n \in \mathcal{A} : n \equiv 1 \pmod{8} \},\\
	\mathcal{A}_2 &= \{ n \in \mathcal{A} : n \equiv 5 \pmod{8} \},\\
	\mathcal{A}_3 &= \{ n \in \mathcal{A} : n \not\equiv 1,5 \pmod{8} \}.
\end{aligned}
\]
Let
\[
\mathcal{A}_i(X) = \#\{ n \in \mathcal{A}_i \ | \ |n| \le X \}.
\]
Then, using the discriminant of pure quartic fields given in \eqref{discrimant_pure_fields}, we obtain:

\begin{equation}\label{B(X)}
\mathcal{B}(X)
= \mathcal{A}_1\!\left( \left(\frac{X}{4}  \right)^{1/3}\right)
+ \mathcal{A}_2\!\left(\left(\frac{X}{16}  \right)^{1/3}\right)
+ \mathcal{A}_3\!\left(\left(\frac{X}{256}  \right)^{1/3}\right).
\end{equation}
Note that every $n\in\mathcal{A}$ admits a unique factorization $n=b\,c^{3}$ with $b$ and $c$ squarefree and coprime.  
This decomposition allows us to obtain an explicit asymptotic formula for $\mathcal{A}(X)$.  
\begin{lemma}\label{count:setA}
	We have
    \begin{equation*}
        \mathcal{A}(X)=C_0\,X+O(\sqrt{X}),
    \end{equation*}
    where
	\[
	C_0 \;=\; \frac{6}{\pi^2}\prod_{p}\left(1+\frac{1}{p(p^2-1)}\right).
	\]
\end{lemma}

\begin{proof}
	Since every $n\in\mathcal A$ can be written uniquely as $n = b\cdot c^3$,
	with $b$ and $c$ relatively prime and squarefree, we can write
	\[
	\cA(X)=\sum_{\substack{c \ \text{squarefree}\\ c^3\le X}}
	\#\{b \le X/c^3 : b \ \text{squarefree},\ (b,c)=1\}.
	\]	
    \noindent
	The inner summand may be counted as
	\begin{align*}
	    \#\{b \le X/c^3 :& b \ \text{squarefree},\ (b,c)=1\} = \sum_{\substack{b \le Y \\ (b,c)=1}} \;\;\sum_{d^2\mid b} \mu(d)\\
        &=\sum_{\substack{d\ge1\\(d,c)=1}} \left(\mu(d)\,\frac{\varphi(c)}{c}\frac{Y}{d^2} + O(1)\right)=\frac{\varphi(c)}{c}\,Y
	\sum_{\substack{d\ge1 \\ (d,c)=1}} \frac{\mu(d)}{d^2}
	+ O(\sqrt{Y})\\
        &  = \frac{6}{\pi^2}\,\frac{\varphi(c)}{c}\,Y
	\prod_{p\mid c}\Big(1-\frac{1}{p^2}\Big)^{-1} + O(\sqrt{Y}).
	\end{align*}
    Taking $Y=\lfloor X/c^3\rfloor$ and summing over $c$, we get that
    \begin{align*}
	\mathcal{A}(X)
	&= \sum_{\substack{c\ \mathrm{squarefree}\\ c^3\le X}}
	\left(
	\frac{6}{\pi^2}\frac{\varphi(c)}{c}\frac{X}{c^3}\prod_{p\mid c}\Big(1-\frac{1}{p^2}\Big)^{-1}
	+ O\!\Big(\sqrt{\frac{X}{c^3}}\Big)
	\right)\\
    & = C_0 X + O(\sqrt{X}),
	\end{align*}
    where
    \begin{align*}
        C_0&= \frac{6}{\pi^2}
	\sum_{\substack{c\ \mathrm{squarefree}}}
	\frac{\varphi(c)}{c^4}\prod_{p\mid c}\Big(1-\frac{1}{p^2}\Big)^{-1}=\frac{6}{\pi^2}\sum_{\substack{c\ \mathrm{squarefree}}} \prod_{p\mid c}\frac{1}{p^3} \left(1-\frac{1}{p}\right) \left(1-\frac{1}{p^2}\right)^{-1}\\
    &= \frac{6}{\pi^2}\sum_{\substack{c\ \mathrm{squarefree}}} \prod_{p\mid c} \frac{1}{p^2(p+1)} = \frac{6}{\pi^2}\prod_p \left(1+\frac{1}{p^2(p+1)}\right),
    \end{align*}
    as required.
\end{proof}

Since $\cA(X)$ has order $X$ and it partitions disjointly into $\cA_1$, $\cA_2$ and $\cA_3$, at least one of $\cA_i(X)$ has order $X$. Hence, from Equation \eqref{B(X)}, we conclude that $\cB(X)$ has order $X^{1/3}$.\\ 

We now count pure quartic fields with a given genus number. For integers $k_{1},k_{2}\geq 0$, define
\[
\cB_{k_{1},k_{2}}(X)
:=
\#\Bigl\{
K=\mathbb{Q}\bigl(\sqrt[4]{a}\bigr)
\ \Big|\ 
\omega_{1}(a)=k_{1},\ \omega_{3}(a)=k_{2}, |\cD_K|\leq X
\Bigr\},
\]
where $\omega_{1}(a)$ and $\omega_{3}(a)$ denote the number of distinct prime divisors
$p\mid a$ with $p\equiv 1\pmod{4}$ and $p\equiv 3\pmod{4}$ respectively.  
For such a field, the genus number equals $2^{k_{1}}4^{k_{2}}$. The correspondence $a\mapsto \mathbb{Q}(\sqrt[4]{a})$ identifies $\mathcal{B}_{k_{1},k_{2}}$
with the set
\[
\mathcal{A}_{k_{1},k_{2}}(X)
:=
\#\Bigl\{
n\in\mathcal{A}(X)
\ \Big|\
\omega_{1}(n)=k_{1},\ \omega_{3}(n)=k_{2}
\Bigr\},
\] 
Thus, counting fields in $\mathcal{B}_{k_{1},k_{2}}$ reduces to counting integers $a\le X$ that are
divisible by exactly $k_{1}$ primes congruent to $1\pmod{4}$ and $k_{2}$ primes congruent to
$3\pmod{4}$, subject to the requirement that $\gcd(\nu,4)=1$ for every $p^{\nu}\parallel a$.

\begin{theorem}\label{thm:upper_bound_Bk1k2}
Fix integers $k_{1},k_{2}\ge 0$ and set $k=k_{1}+k_{2}\ge 1$.  
Then, as $X\to\infty$,
\[
\mathcal{B}_{k_{1},k_{2}}(X)
= O\!\left(\frac{X}{(k-1)!\log X}(\log\log X)^{\,k-1}\right).
\]
\end{theorem}

\begin{proof}
By the correspondence between $\cA(X)$ and $\cB(X)$, we have
\[
\mathcal{B}_{k_1,k_2}(X)
\ \ll\ 
\mathcal{A}_{k_1,k_2}\!\left( \left(\frac{X}{4}\right)^{1/3} \right).
\]
For fixed nonnegative integers $k_1,k_2$, define
\[
\mathcal{A}_{k_1,k_2}(X)
:= 
\Bigl\{
a \in \mathcal{A} : |a|\le X,\ 
\#\{p\mid a : p\equiv 1 \pmod{4}\} = k_1,\
\#\{p\mid a : p\equiv 3 \pmod{4}\} = k_2
\Bigr\}.
\]
Since the primes dividing~$a$ are partitioned by their residue class modulo~$4$, 
we clearly have
\[
\#\mathcal{A}_{k_1,k_2}(X)
\ \ll\
\#\{a \le X : \omega(a)=k_1+k_2\},
\]
where $\omega(a)$ denotes the number of distinct prime factors of $a$. By a classical result of Landau (see \cite[Chapter II.6]{Ten15book}), it is known that, for any $k\geq 1$
\[
\#\{a \le X : \omega(a)=k\}
\ll \frac{X}{(k-1)!\,\log X} (\log\log X)^{k-1}.
\]
Therefore
\[
\mathcal{B}_{k_1,k_2}(X)
\ll \frac{X^{1/3}}{\log X} ((k_1+k_2-1)!\,\log\log X)^{k_1+k_2-1},
\]
as claimed.
\end{proof}

\medskip
Since $|\cB(X)|\gg X^{1/3}$, we deduce that in the family of pure quartic fields, any prescribed genus number is attained with density zero, proving Theorem \ref{pure-quartic}. More precisely, 
\[
   \lim_{X\to\infty} \frac{\#\{K \in \mathcal{B}(X): g_K = 2^{t},\ |\mathcal{D}_K|\le X\}}
         {\mathcal{B}(X)}
    = 0.
\]

\medskip
\section{\bf Mean and higher moments for genus numbers}\label{higher-moments}
\medskip

In \cite[Theorem 2]{cubicwithgenusnumber1}, McGown and Tucker computed the average of the genus number over cubic fields. They showed that
\begin{equation}\label{tucker-mean}
    \sum_{\substack{K\ \text{cubic field} \\ \mathcal{D}_K \le X}} g_K
\;=\;
S_0
X \;+\;
O\!\left(X^{16/17+\varepsilon}\right),
\end{equation}
where \begin{align*}
  S_0= \frac{119}{324\,\zeta(2)}
\!\left(\prod_{p\equiv 1 \!\!\!\pmod 3}
      \Bigl(1+\frac{3}{p(p+1)}\Bigr)\right) &
\!\left(\prod_{p\equiv 2 \!\!\!\pmod 3}
      \Bigl(1+\frac{1}{p(p+1)}\Bigr)\right).
\end{align*}
We employ their method to obtain higher moments for the genus number over cubic fields, thus giving a complete picture of the genus distribution over cubic fields. A similar computation for higher moments has also been carried out by Yamada \cite{Yam26}.
\begin{proposition}\label{prop:highermoments:S3}
For a positive integer $r$, as $X$ tends to infinity
\[
\sum_{\substack{K\in \cM(S_3, 3, X)}} g_K^{\,r}
\;=\;
S_0^{(r)}\,X \;+\; O\!\bigl(X^{16/17+\varepsilon}\bigr),
\]
where
\[
S_0^{(r)} \;=\; \frac{116+3^{r}}{324\,\zeta(2)}
\prod_{p\equiv1\pmod3}\Bigl(1+\frac{3^{r}}{p(p+1)}\Bigr)
\prod_{p\equiv2\pmod3}\Bigl(1+\frac{1}{p(p+1)}\Bigr).
\]
\end{proposition}
\begin{proof}
    The proof follows a similar method as in \cite[section 4] {cubicwithgenusnumber1}. For squarefree $f$ coprime to $3$, let $N(f; X)$ denote the number of cubic fields with $|\cD_K|\le X$ that are totally ramified exactly at the primes dividing $f$ (and no other primes). Let $N'(f;X )$ denote the subcount of those fields for which in addition $3$ is totally ramified and $d\equiv1\pmod3$. 
		Using the two possible genus values, we obtain the identity
		\begin{align}\label{eq:Sr-decomp}
		\sum_{\substack{K\in \cM(S_3, 3, X)}}  g_K^{\,r}
	&	=  \sum_{\cD_K\leq X} 3^{r\psi(f)}
		- \sideset{}{'}\sum_{\cD_K\leq X} 3^{r\psi(f)}
		+ \sideset{}{'}\sum_{\cD_K\leq X} 3^{r(\psi(f)+1)},
	\end{align}
	where $\sideset{}{'}\sum$ is the sum over  cubic fields where $3$ is totally ramified and with $d\equiv 1 \pmod{3}$. Rewriting the expression, by grouping together squarefree conductor $f$, with $(f,3)=1$, we have
    \[
        \sum_{\substack{K\in \cM(S_3, 3, X)}} g_K^{\,r}
        = \sum_{\substack{(f,3)=1\\ \\ f \text{squarefree} }} 3^{r\psi(f)}\, N(f;X)
\;+\;
(3^{r}-1) \sum_{\substack{(f,3)=1\\ f ~\text{squarefree}}} 3^{r\psi(f)}\, N'(f;X).
\]
We now invoke Proposition 3.1 and 3.2 of \cite{cubicwithgenusnumber1}, which gives
\begin{align*}
		N(f;X)
		&= \frac{13}{36\zeta(2)}
		\Bigl(\prod_{p\mid f}\frac{1}{p(p+1)}\Bigr)X
		+ O\!\bigl(f^{-1}X^{16/17+\varepsilon}\bigr) \,\,\text{ and}\\[4pt]
		N'(f;X)
		&= \frac{1}{324\,\zeta(2)}
		\Bigl(\prod_{p\mid f}\frac{1}{p(p+1)}\Bigr)X
		+ O\!\bigl(f^{-1}X^{16/17+\varepsilon}\bigr). 
	\end{align*}
    Using this in \eqref{eq:Sr-decomp}, we deduce that
    \[
\sum_{\substack{K\in \cM(S_3, 3, X)}} g_K^{\,r}
= S_0^{(r)}X
\;+\; O(X^{16/17 + \varepsilon}).
\]
    
\end{proof}
\smallskip

For each fixed integer $r \ge 0$, the limiting $r$-th moment of $g_K$ over cubic fields is given by
\[
\mu_r(S_3,3) \;:=\; \lim_{X\to\infty}
\frac{\displaystyle \sum_{\substack{K\in \cM(S_3, 3, X)}} g_K^{\,r}}
     {\displaystyle |\cM(S_3, 3, X)|}.
\]
By Proposition \ref{prop:highermoments:S3}, this limit exists and is given by
\[
\mu_r(S_3,3) = 3\zeta(3)\,S_0^{(r)},
\]
where
\[
S_0^{(r)} \;=\; \frac{116+3^{r}}{324\,\zeta(2)}
\prod_{p\equiv1\pmod3}\!\Bigl(1+\frac{3^{r}}{p(p+1)}\Bigr)
\prod_{p\equiv2\pmod3}\!\Bigl(1+\frac{1}{p(p+1)}\Bigr).
\]
\smallskip

A similar analysis for the higher moments can also be carried out for $S_3\times C_q$- fields, where $q\neq 3$ is a prime number. Recall that for a prime $q$
\[
\Psi_q(a)
  \;:=\;
  \#\Bigl\{
      p \mid a
      \;\Bigm|\;
      p \equiv 1 \!\!\pmod{q}
      \ \text{or}\
      p = q
    \Bigr\}.
\]
\begin{proposition}\label{higher-moments-S_3XC_q}
Let $q$ be a prime number $\neq  3$. For a positive integer $r$, as $X$ tends to infinity
    \[
\displaystyle
\sum_{L\in\cM(S_3\times C_q, 3q, X)}
g_L^{\,r}
   \;=\;
 {S_0^{(r)} S_q^{(r)}}\,X^{1/q}\,
   \;+\; O\!\left(X^{16/17q+\varepsilon}\right),
\]
where $S_0^{(r)}$ is as above and
\[
S_q^{(r)}
   \;:=\;
   \sum_{\substack{[F:\Q]=q \\ \Gal(F/\Q)\simeq C_q}}
  \frac{q^{ r \left( \Psi_q(\cD_F)-1\right)}}{\cD_F^{3/q}}
\]
\end{proposition} 
\begin{proof}
    For $q \ge 5$, every $L \in \cM(S_3 \times C_q, 3q, X)$ can be written as $L = KF$, where $K$ and $F$ are linearly disjoint with $\Gal(K/\Q) \cong S_3$ and $\Gal(F/\Q) \cong C_q$. For $q = 2$, Proposition \ref{Count-S_3xC_2} shows that the same holds for almost all $L \in \cM(S_3 \times C_2, 6, X)$.  Hence
    \begin{equation*}
        \sum_{\substack{L \in \cM(S_3\times C_q, 3q, X)}} g_L^r = \sum_{\substack{F \in \cM(C_q, q, X^{1/3})}} g_F^r \sum_{\substack{K \in \cM\left(S_3, 3, \left(\frac{X}{\cD_F^3}\right)^{1/3}\right)}} g_K^r.
    \end{equation*}
    Using Proposition \ref{prop:highermoments:S3}, we deduce that
    \begin{align*}
        \sum_{\substack{L \in \cM(S_3\times C_q, 3q, X)}} g_L^r &=
        \sum_{\substack{F \in \cM(C_q, q, X^{1/3})}} g_F^r  \left( S_0^{(r)}
	\bigl(X/\mathcal{D}_F^{3}\bigr)^{1/q} \;+\;
	O\!\left(X^{16/17q\,+\varepsilon}\right)
	\right)\\
    &= S_0^{(r)} X^{1/q} \sum_{F\in \cM(C_q,q, X^{1/3})}
\frac{g_F^r}{\mathcal{D}_F^{\,3/q}} \;+\;
O\bigl(X^{\frac{16}{17q}+\varepsilon}\bigr)\\
    & = S_0^{(r)} S_q^{(r)}\,X^{1/q}
\;+\;
O\bigl(X^{\frac{16}{17q}+\varepsilon}\bigr). 
    \end{align*}
    \end{proof}
\smallskip    

By considering the primes $p \le 5\times 10^5$ in the Euler product above, we obtain the following numerical values, which suggests extreme skewed behaviour of genus numbers over $S_3$ and $S_3\times C_2$-fields away from $1$.

\begin{center}
\begin{tabular}{|c|c|c|}
\hline
$r$ & $\mu_r(S_3,3)$ & $\mu_r(S_3\times C_2,6)$\\ 
\hline
0 & 1.000000 & 1.000000 \\
1 & 1.078541 &  2.61129\\
2 & 1.340801 & 4.46212\\
3 & 2.409905 &21.775405\\
4 & 9.659183 &627.98246 \\
5 & 153.99614& 276821.108977 \\
\hline
\end{tabular}
\end{center}
\medskip

\section{\bf Heuristics on genus zero density}\label{heuristics}
\noindent
{\bf $D_4$-Octic fields $\cM(D_4, 8)$}: For any $L\in \cM(D_4, 8)$, if a rational prime $p$ is ramified $L$ then $e(p)\geq2$. Therefore,  by Theorem \ref{lemma:k1k2decomp}, we have that each prime dividing $\cD_L$, contributes alteast a factor of $2$ to its genus number. Consequently, if $g(L)=2^n$, then
\[
\omega(\mathcal{D}_L)\le n+1.
\]
Hence,
\begin{equation*}
    \sum_{\substack{L\in \cM(D_4,8),\\ |\cD_L|\leq X,\\ g_L\leq 2^n}} 1 \leq   \sum_{\substack{L\in \cM(D_4,8),\\ |\cD_L|\leq X,\\ \omega(\cD_L)=n+1}} 1.
\end{equation*}
For any fixed $n$, heuristically, one would expect that the set of octic $D_4$-fields whose discriminant has at most $n+1$ prime divisors has density zero and hence so is the number of such fields with genus number $g(L)=2^n$. One can in fact, give a similar heuristic for any family of Galois fields with Galois group $2^m$.\\

\noindent
\textbf{Heuristic for genus one $C_7\rtimes C_3$-fields}: Let $L$ be a degree $21$ number field whose Galois closure has group  $G_{21}\cong C_7 \rtimes C_3$.
Write $L=KF$, where $F$ is the cubic subfield and $K$ is the degree $7$ field whose Galois closure has group $C_7\rtimes C_3$. For such fields, the genus number satisfies
\[
g(L)=3^{t_3-1}\,7^{t_7},
\]
where $t_3$ denotes the number of primes ramified in the cubic field $F$ and $t_7$ denotes the number of primes ramified in $K$ with inertia group $C_7$ and satisfying $p\equiv 1 \pmod{7}$.\\
\noindent

Recall that the contribution of a ramified prime $p$ to the genus number is $\gcd(e(p),p-1)$. Thus, if $p\mid \mathcal{D}_F$, then $e(p)=3$ and necessarily $p\equiv 1\pmod{3}$. Hence $\gcd(3,p-1)=3$ and each such prime contributes a factor $3$. However, if $p$ ramifies in $K$ with inertia group $C_7$, then its contribution to the genus number is $\gcd(7,p-1)$, which equals $7$ precisely when $p\equiv1\pmod7$.\\

If $g(L)=1$, then  $t_3=1$ and  $t_7=0$. In other words, exactly one prime ramifies in $F$ and no prime is ramified with inertia group $C_7$ satisfying $p\equiv 1 \pmod 7$. Thus, we restrict ourselves to counting only those number fields $L$, whose discriminant $\cD_L$ is not divisible by any prime $p$ in the congruence class $1\pmod 7$. Heuristically, such fields should have density zero, since the set of integers whose prime factors avoid a fixed congruence class has density zero. In other words, heuristically as $X\to\infty$
$$
    \frac{\#\{L\in\mathcal{F}:|\mathcal{D}_L|\le X,\ g(L)=1\}}
{\#\{L\in\mathcal{F}:|\mathcal{D}_L|\le X\}}
\longrightarrow 0.
$$

Both the above examples are in support of Conjecture \ref{conj-1}. It is perhaps possible to explicitly establish these heuristics under Malle's conjecture. We relegate it to future investigation.

 \section*{\bf Acknowledgments}
\medskip

We thank Prof. Frank Thorne and Prof. Takashi Taniguchi for their helpful comments on an earlier version of this article. We are also grateful to Prof. Daniel Loughran and Prof. Arul Shankar for their valuable suggestions. This research was initiated during the thematic program on arithmetic statistics at the Lodha Mathematical Sciences Institute (LMSI) in Mumbai, and we warmly acknowledge the Institute for its stimulating and hospitable research environment. The second author is also grateful to the Max Planck Institute for Mathematics for its hospitality, where part of this research was undertaken.

\end{document}